\documentclass[11pt,a4paper]{article} 
\usepackage[latin1]{inputenc}  

\usepackage{amsmath,amscd,amssymb,latexsym,graphicx,color,mathrsfs} 
 
 \usepackage[all]{xy}

\newtheorem{theorem}{Theorem}[section]

\newtheorem{proposition}[theorem]{Proposition}
\newtheorem{corollary}[theorem]{Corollary}

\newtheorem{lemma}[theorem]{Lemma}
\newtheorem{definition}[theorem]{Definition}
\newtheorem{remark}[theorem]{Remark}

\numberwithin{equation}{section} 
 
\newcommand{\cqfd}{\hfill{\small $\Box$}} 
\newenvironment{proof}[1][]{{\bf Proof #1 : }}{\begin{flushright} 
\cqfd\end{flushright}}

\reversemarginpar

\newcommand{\ubd}{\operatorname{ubd}}
\newcommand{\ad}{\operatorname{ad}}

\newcommand{\gr}{\mathscr{G}}
\newcommand{\go}{\mathscr{G} ^{(0)}}
\newcommand{\hr}{\mathscr{H}}
\newcommand{\ho}{\mathscr{H} ^{(0)}}
\newcommand{\lr}{\mathscr{L}}

\newcommand{\Nb}{\mathscr{N}}

\def\intx{\overset{\:\circ}{X}}
\newcommand\tgt[1]{{}^{T}\kern-1pt #1}
\newcommand\adi[1]{{}^{ad}\kern-1pt #1}

\newcommand{\ci}{C_{c}^{\infty}}

  \def\RR{{\mathrm{R}}}

  \def\CC{{\mathbb{C}}}

  \def\RR{{\mathbb{R}}}

 \def\ZZ{{\mathbb{Z}}}

 \def\cE{{\mathcal{E}}} \def\cF{{\mathcal{F}}}
  
 \def\cK{{\mathcal{K}}}

 \def\cT{{\mathcal{T}}} \def\cU{{\mathcal{U}}}
\def\cV{{\mathcal{V}}} \def\cW{{\mathcal{W}}} 
 
\newcommand{\T}{\mathcal{T}_{nc}X} 

\newcommand\FCX{{}^\mathsf{FC}\!X}

\title{A cohomological formula for the Atiyah-Patodi-Singer index on manifolds with boundary}  
 
\author{P. Carrillo Rouse, J.M. Lescure and B. Monthubert} 
 
\date{\today} 
 
\begin{document}

\maketitle 
\begin{center}
{\bf Abstract}
\end{center}
 The main result of this paper is a new Atiyah-Singer type cohomological formula for the index of Fredholm pseudodifferential operators on a manifold with boundary. The non locality of the chosen boundary condition prevents us to apply directly the methods used by Atiyah and Singer in \cite{AS,AS3}. However, by using the $K$-theory of $C^*-$algebras associated to some groupoids, which generalizes the classical $K$-theory of spaces, we are able to understand the computation of the APS index using classic algebraic topology methods ($K$-theory and cohomology).
As in the classic case of Atiyah-Singer (\cite{AS,AS3}), we use an embedding into an euclidean space to express the index as the integral of a true form on a true space, the integral being over a $C^\infty$-manifold called the singular normal bundle associated to the embedding.
Our formula is based on a K-theoretical Atiyah-Patodi-Singer theorem for manifolds with boundary that is inspired by Connes' tangent groupoid approach, it is not a groupoid interpretation of the famous Atiyah-Patodi-Singer index theorem.

\tableofcontents

\let\thefootnote\relax

\section{Introduction}

In the early 60's, Atiyah and Singer gave a positive answer to a problem posed by Gelfand about investigating the relationship between topological and analytical invariants of elliptic (pseudo)differential operators on closed smooth manifolds without boundary, \cite{AS,AS3}. In a series of papers, Atiyah-Singer not only gave a general cohomological formula for the index of an elliptic (pseudo)differential operator on a closed smooth manifold, they also gave several applications and more importantly they opened a entire new way of studying index problems. Since then, index theory has been at the core of interest of several domains in mathematics and mathematical physics.

To be more descriptive, let $M$ be a closed smooth manifold of dimension $n$, let $D$ be a elliptic (pseudo)differential operator with principal symbol $\sigma_D$. The Atiyah-Singer index formula states
\begin{equation}\label{ASindexformula}
ind\, D=(-1)^n\int_{T^*M}ch([\sigma_D])\cT(M)
\end{equation}
where $[\sigma_D]\in K^0(T^*M)$ is the principal symbol class in K-theory, $ch([\sigma_D])\in H^{ev}_{dR}(T^*M)$ its Chern character, $\cT(M)\in H^{ev}_{dR}(M)$ the Todd class of (the complexified) $T^*M$ and $T^*M$ is oriented as an almost complex manifold, following \cite{AS3}.

A fundamental step in order to achieve such a formula was to realize that the map $D\mapsto ind\, D$ is completely encoded by a group morphism
$K^0(T^*M)\longrightarrow \mathbb{Z},$ 
called the analytic index of $M$. That is, if $Ell(M)$ denotes the set of elliptic pseudodifferential operators over $M$, then the following diagram is commutative:
\begin{equation}\label{elldiam}
\xymatrix{
Ell(M)\ar[r]^-{ind} \ar[d]_-{\sigma}& \mathbb{Z}\\
K^0(T^*M)\ar[ru]_-{ind_{a,M}} & ,
}
\end{equation}
where $Ell(M)\stackrel{\sigma}{\longrightarrow}K^0(T^*M)$ is the surjective map that associates to an elliptic operator the class of its principal symbol in $K^0(T^*M)$. The use of K-theory was a breakthrough in the approach by Atiyah-Singer, indeed, they could use its (generalized) cohomological properties to decompose the analytic index morphism in a composition of topologic (and hence computable) morphisms. The idea is as follows. Consider an embedding $M\hookrightarrow \RR^N$ (assume $N$ even for the purpose of this exposition) and the corresponding normal bundle $N(M)$, Atiyah-Singer showed that the analytic index decomposes as the composition of
\begin{itemize}
\item The Thom isomorphism\footnote{The composition (with obvious notations) $T^*M\stackrel{\pi}{\to}M\stackrel{i_0}{\to}N(M)$ is K-oriented.}
$$K^0(T^*M)\stackrel{\mathscr{T}}{\longrightarrow}K^0(N(M))$$ 
followed by
\item the canonical morphism
$$K^0(N(M))\stackrel{j!}{\longrightarrow} K^0(\RR^N)$$
induced from a identification of the normal bundle as an open subset of $\RR^N$, and followed by
\item the Bott isomorphism
$$K^0(\RR^N)\stackrel{B}{\longrightarrow}K^0(\{pt\})\approx \mathbb{Z}.$$
\end{itemize}

In particular, modulo the Thom and Bott isomorphisms, the analytic index is transformed in a very simple shriek map: $K^0(N(M))\stackrel{j!}{\longrightarrow} K^0(\RR^N).$ The formula (\ref{ASindexformula}) is then obtained as an algebraic topology exercise of comparison between K-theory and cohomology, \cite{AS3}. 

For the purposes of the present paper, remark that the formula (\ref{ASindexformula}) can be also written as follows:
\begin{equation}\label{ASindexformulanormal}
ind\, D=\int_{N(M)}ch(\mathscr{T}([\sigma_D])).
\end{equation}

To get into the subject of the present article we need to talk about groupoids. In his book, \cite{Concg} Section II.5, Connes sketches a (conceptually) simple proof of the K-theoretical Atiyah-Singer Index theorem for closed smooth manifolds using tangent groupoid techniques. The idea is the following: let $M$ be a
closed smooth manifold and $G_M=M\times M$ its pair groupoid. For the readers not familiar with groupoids, one can think of the kernel algebra (convolution algebra of the groupoid) for the pair groupoid: that is on the algebra of smooth complex valued functions on $M\times M$ with kernel convolution product.

Consider the tangent groupoid 
$$\tgt{G_M}:=TM\times \{0\}\bigsqcup M\times M\times (0,1]\rightrightarrows M\times [0,1].$$ 
It is nowadays well known that the index morphism provided by this deformation groupoid is precisely the analytic index of Atiyah-Singer, \cite{Concg,MP}. In other words, there is a short exact sequence of $C^*$-algebras
\begin{equation}\label{tangentsuite}
\xymatrix{
0\ar[r]&C^*(M\times M\times (0,1])\ar[r]&C^*(\tgt{G_M})\ar[r]^-{e_0}&C_0(T^*M)\ar[r]&0
}
\end{equation}
and since $C^*(M\times M\times (0,1])$ is contractible, the morphism induced in K-theory by $e_0$ is invertible. The analytic index of $M$ is the morphism
\begin{equation}\label{index}
\xymatrix{
K_0(C^*(TM))\ar[r]^-{(e_0)_{*}^{-1}}&K_0(C^*(\tgt{G_M}))\ar[r]^-{(e_1)_*}&K_0(C^*(M\times M))
=K_0(\mathscr{K}(L^2(M)))\approx \mathbb{Z},
}
\end{equation}
where $e_t$ are the obvious evaluation  morphisms at $t$.  

Actually, all groupoids considered in this work are (at least) continuous family groupoids, so there is a notion of reduced and envelopping  $C^*$-algebra (using half-densities or by picking up a continuous Haar system) and amenable as well, thus the distinction between the reduced and envelopping $C^*$-algebras is not even necessary.

As pointed out by Connes, if the groupoids appearing in this interpretation of the index were Morita equivalent to spaces then we would immediately have a geometric
interpretation of the index. Now, $M\times M$ is Morita equivalent to a point
(hence to a space), but the other fundamental groupoid playing a role in the previous discussion is not, that is, $TM$ is a groupoid whose fibers are the groups $T_xM$, which are  not Morita equivalent (as groupoids) to  spaces. The idea of Connes is to use an appropriate action of the tangent groupoid in some $\mathbb{R}^N$ in order to translate the index (via a Thom isomorphism) in an index associated to a deformation groupoid which will be Morita equivalent to some space.



\vspace{2mm}

{\bf The case of manifolds with boundary}

\vspace{2mm}

In a series of papers \cite{APS1,APS2,APS3}, Atiyah, Patodi and Singer investigated the case of 
non-local elliptic boundary value problems. They showed that under
some boundary conditions (the now so-called APS boundary conditions),
Dirac operators (among a larger class of first order differential operators) become Fredholm on suitable spaces and they computed
the index. To the characteristic form from the closed smooth case they
added a correction term, called the eta invariant, which is determined by an
appropriate restriction of the operator to the boundary : this  is a
spectral invariant, measuring the asymmetry of the
spectrum. However, a cohomological formula expressing
the APS index  as an integration of some characteristic form on the boundary is impossible. Roughly speaking, the non locality of the chosen boundary condition prevents us to express this spectral correction term with a local object, like a density on the boundary. In other words, applying directly the methods of Atiyah-Singer cannot yield the result in the APS index problem. However, we will see that the use of some noncommutative spaces, associated with groupoids, will give us the possibility to understand the computation of the APS index using classic algebraic topology methods ($K$-theory and cohomology).

\vspace{2mm}
In this paper, we will follow Connes' groupoid approach to obtain a
cohomological formula for the index of a fully elliptic
(pseudo)differential operator on a closed manifold with boundary. For
the case of such a manifold, the pair groupoid does not give the same
information as the smooth case. In fact this case of manifolds with boundary becomes more interesting since different boundary conditions can be considered and each of these give different index problems. In this paper we will be interested in the so-called Atiyah-Patodi-Singer boundary condition. For the moment we will not recall what this condition is, in fact we rather describe the groupoid whose pseudodifferential calculus gives rise to the index theory related to such a condition.

Let $X$ be a manifold with boundary. We denote, as usual, $\intx$ the interior which is a smooth manifold and $\partial X$ its boundary. Let 
\begin{equation}
\Gamma(X)\rightrightarrows X
\end{equation}
be the groupoid of the b-calculus, where 
$$\Gamma(X)=\intx\times\intx \bigsqcup \partial X\times \partial X\times \RR,$$ 
with groupoid structure given as a family of pair groupoids and the (additive) group $\RR$. It is a continuous family groupoid with the topology explicity described in \cite{Mont} (see beginning of Section \ref{nctspacesection} below) 

Consider 
$$\Gamma(X)^{tan}=A(\Gamma(X))\bigsqcup \Gamma(X)\times (0,1]\rightrightarrows X\times [0,1]$$ 
its tangent groupoid.

Take now the open subgroupoid of $\Gamma(X)^{tan}$ obtained by restriction to $
X_\cF:=X\times [0,1]\setminus (\partial X\times \{1\})$
\begin{equation}\nonumber
\Gamma(X)_{\cF}=A(\Gamma(X))\times \{0\}\bigsqcup \intx\times\intx\times (0,1]\bigsqcup \partial X\times \partial X\times \RR\times (0,1)\rightrightarrows X_\cF.
\end{equation}

By definition $\intx\times \intx \times (0,1]$ is a saturated, open
dense subgroupoid of $\Gamma(X)_\cF$. This leads to a complementary closed subgroupoid of $\Gamma(X)_\cF$:
\begin{equation}
\mathcal{T}_{nc}X\rightrightarrows X_{\partial},
\end{equation}
where $X_{\partial}:=X_\cF\setminus \intx \times (0,1]=X\underset{\partial X\times\{0\}}{\cup}\partial X\times [0,1)$.

The groupoid $\mathcal{T}_{nc}X$, called here {\it "The noncommutative
  tangent space of $X$"}, was introduced  in \cite{DL05} in the framework of pseudomanifolds with isolated singularities and used for Poincar\'e duality purpose. It was later used again in \cite{DLN} to derive an index theorem and reintroduced in \cite{Mont-Banach} in the framework of manifolds with boundary. Note also that $\mathcal{T}_{nc}X$ is denoted by $T\FCX$ in \cite{DLR} where it is generalized to the case of stratified spaces, or equivalently, manifolds with (iterated) fibred corners. 

Deformation groupoids like $\Gamma(X)_{\cF}$ induce index morphisms. Indeed, its algebra comes 
equipped with a restriction morphism to the algebra of $\mathcal{T}_{nc}X$ and an evaluation morphism to the algebra of $\intx\times \intx$ (for $t=1$). Indeed, we have a short exact sequence of $C^*$-algebras
\begin{equation}\label{btangentsuiteintro}
\xymatrix{
0\ar[r]&C^*(\intx\times \intx\times (0,1])\ar[r]&C^*(\Gamma(X)_{\cF})\ar[r]^-{e_0}&C^*(\mathcal{T}_{nc}X)\ar[r]&0
}
\end{equation}
where the algebra $C^*(\intx\times \intx\times (0,1])$ is contractible. Hence applying the $K$-theory functor to this sequence we obtain an index morphism
\begin{equation}\label{indfintro}
ind_{\cF}=(e_1)_*\circ(e_0)_{*}^{-1}:K_0(C^*(\mathcal{T}_{nc}X))\longrightarrow
K_0(C^*(\intx\times \intx))\approx \mathbb{Z}.
\end{equation}
This index computes indeed the Fredholm index of those elliptic operators on X satisfying the APS boundary condition, and hence we call it {\it The Fredholm index morphism of $X$}. To be more explicit, the statement is the following: 

\begin{proposition}\label{FAPSintro}
For any fully elliptic operator $D$ on $X$, there is a naturally associated "non commutative symbol" $[\sigma_D]\in K_0(C^*(\mathcal{T}_{nc}X))$ and
\begin{equation}
ind_{\cF}([\sigma_D])= Index_{APS}(D),
\end{equation}
where $Index_{APS}(D)$ is the Fredholm index of $D$. Moreover, every element in $K_0(C^*(\mathcal{T}_{nc}X))$ can be realized in this way.
\end{proposition}
Previous related results appeared in \cite{LMNpdo} for differential operators and using different algebras to classify their symbols, and in \cite{Savin2005} where different techniques were employed. Proposition \ref{FAPSintro} was proved in \cite{Les09} and the corresponding statement for general manifolds with fibred corners was treated in \cite{DLR}.

A first task in order to follow the Atiyah-Singer approach would be to
compute the morphism $ind_{\cF}$ by topological means. For instance, using
an appropriate embedding into a space in which the computation could
follow in an easier way. This idea has been already followed up in \cite{DLN} in the framework of manifolds with conical singularities, using a $KK$-equivalent version of the noncommutative tangent space $\mathcal{T}_{nc}X$. There, the authors use embeddings into euclidean spaces to extend the construction of the Atiyah-Singer topological index map, thanks to a ``Thom isomorphism`` between the noncommutative tangent space of the singular manifold and of its singular normal bundle, and then get an index theorem in the framework of $K$-theory. Here, we follow a different approach and we are going to extend  the Atiyah-Singer topological index map using Connes' ideas on tangent groupoid actions on euclidean spaces; moreover we investigate the cohomological counterpart of the $K$-theoretic statement of the index theorem. Note also that the index map considered here coincide, through $KK$-equivalences, with the index maps considered in \cite{LMNpdo} and in \cite{DLN}.

We start in Section \ref{sectionembedding} by considering an appropriate embedding ({\it i.e.}, respecting the boundary)
\begin{equation}\label{embeddingintro}
i:X\hookrightarrow \mathbb{R}^{N-1}\times \mathbb{R}_+
\end{equation}
of $X$ into $\mathbb{R}^{N-1}\times \mathbb{R}_+$. Following Connes, we use it to define a continuous family groupoid morphism (see Section \ref{hdefinition} for the explicit definition)
\begin{equation}
h:\Gamma(X)\to \RR^N
\end{equation}
where we see $\RR^N$ as an additive group and we assume $N$ even. This morphism induces an action of $\Gamma(X)$ on $X\times \RR^N$ and an induced deformation action of $\Gamma(X)_\cF$ on $X_\cF\times \RR^N$ (coming from an induced morphism $h_\cF$). The main task in Section \ref{sectionembedding} is to prove the following result (proposition \ref{hcontinue} below):

\begin{proposition}\label{hcontinueintro}
The semi-direct groupoid $(\Gamma(X)_{\cF})_{h_\cF}:=\Gamma(X)_\cF\rtimes_{h_\cF} \RR^N$ is a free proper groupoid.
\end{proposition}

In Section \ref{CTsection} we explain how the Connes-Thom isomorphism links the K-theory of a groupoid with the K-theory of a semi-direct groupoid as above. For instance, the new semi-direct groupoid $(\Gamma(X)_{\cF})_{h_\cF}$ defines as well an index morphism and this one is linked with the index (\ref{indfintro}) by a natural isomorphism, the so called Connes-Thom isomorphism, thus giving the following commutative diagram
\begin{equation}\label{CTdiagramintro}
\xymatrix{
K_0(C^*(\cT_{nc}X))\ar@/^3pc/[rr]^-{ind_\cF}\ar[d]_-{\mathscr{CT}}^-{\approx}&K_0(C^*(\Gamma(X)_{\cF}))\ar[d]_-{\mathscr{CT}}^-{\approx}\ar[l]_-{e_0}^-{\approx}\ar[r]^-{e_1}&K_0(C^*(\intx\times\intx))\approx \ZZ\ar[d]_-{\mathscr{CT}}^-{\approx}\\
K_0(C^*((\cT_{nc}X)_{h_0}))\ar@/_3pc/[rr]_-{ind_{h_\cF}}&K_0(C^*((\Gamma(X)_{\cF})_{h_\cF}))
\ar[l]_-{e_0}^-{\approx}\ar[r]^-{e_1}&K_0(C^*((\intx\times\intx)_{h_1}))
}
\end{equation}
where $h_0$ and $h_1$ denote the respective restricted actions of $\cT_{nc}X$ and $\intx\times\intx$ on $\RR^N$.

Now, the proposition above tells us that the orbit space of
$(\Gamma(X)_{\cF})_{h_\cF}$ is a nice space and moreover that this
semi-direct groupoid is Morita equivalent to its  orbit
space. This means that the index morphism $ind_{h_\cF}$ can be
computed, modulo Morita equivalences, as the deformation index morphism of some
space. More precisely, denoting by $\mathscr{B}_{h_\cF}$ the orbit space of $(\Gamma(X)_{\cF})_{h_\cF}$, by $\mathscr{B}_{h_0}$ the orbit space of $(\cT_{nc}X)_{h_0}$ and by $\mathscr{B}_{h_1}$ the orbit space of $(\Gamma(X)_{\cF})_{h_\cF}$ we have an index morphism between K-theory of spaces (topological K-groups, no more $C^*$-algebras if one does not like it!)
\begin{equation}\label{Bdiagramintro}
\xymatrix{
ind_{\mathscr{B}_{h_\cF}}:&K^0(\mathscr{B}_{h_0})&K^0(\mathscr{B}_{h_\cF})
\ar[l]_-{e_0}^-{\approx}\ar[r]^-{e_1}&K^0(\mathscr{B}_{h_1})
}
\end{equation}
from which we would be able to compute the Fredholm index. This is
what we achieve next, indeed, in Section \ref{BAPSsection} where we are able to explicitly identify these orbit spaces. 

In order to describe them we need to introduce a new space, but let
us first motivate this by looking at the situation when $\partial
X=\emptyset$ (following \cite{Concg} II.5). In this case, the orbit space of $(\Gamma(X)_{\cF})_{h_\cF}$ can be identified with the deformation to the normal cone (see appendix \ref{dncappendix} below for the $C^\infty$-structure of such deformations) of the embedding $X\hookrightarrow \RR^N$, that is a $C^\infty$-cobordism between the normal bundle to $X$ in $\RR^N$, $N(X)$, and $\RR^N$ itself:
\begin{equation}\label{BASintro}
\mathscr{B}_{AS}:=N(X)\bigsqcup (0,1]\times\RR^N.
\end{equation}

In this picture we also see the orbit space of $(TX)_{h_0}$ which identifies with $N(X)$ and the orbit space of $(X\times X)_{h_1}$ which identifies with $\RR^N$.

Still, in this boundaryless case ($\partial X=\emptyset$), this space $\mathscr{B}_{AS}$ gives in K-theory a deformation index morphism
$$\xymatrix{
ind_{\mathscr{B}_{AS}}:&K^0(N(X))&K^0(\mathscr{B}_{AS})\ar[l]^-{e_0}_-{\approx}\ar[r]^-{e_1}& K^0(\RR^N)
}$$
which is easily seen to be the shriek map associated to the identification of $N(X)$ as an open subset of $\RR^N$. 

In the boundary case, the normal bundle is not the right space, we
know for instance that the APS index cannot be computed by an
integration over this space due to the non-locality of the APS boundary condition. One has then to compute the orbit spaces, in fact the orbit space of $(\cT_{nc}X)_{h_0}$ identifies (lemma \ref{lemmaqbord}) with the singular normal bundle:
\begin{equation}\label{Nsingintro}
\mathscr{N}_{sing}(X):=N(X)\times \{0\}\bigsqcup \RR^{N-1}\times (0,1)
\end{equation}
which is the $C^\infty$-manifold obtained by gluing the normal bundle $N(X)$\footnote{which is an honest vector bundle over $X$.} associated to the embedding (\ref{embeddingintro}) and
$$D_\partial:=N(\partial X)\times \{0\}\bigsqcup \RR^{N-1}\times (0,1)$$
the deformation to the normal cone associated to the embedding 
$\partial X\hookrightarrow \RR^{N-1}$, along their common boundary (the gluing depending on a choice of a defining function of the boundary of $X$). The orbit space of $(\intx\times\intx)_{h_1}$ is easily identified with $\RR^N$ (lemma \ref{lemmaqint}). 

Finally the orbit space of $(\Gamma(X)_{\cF})_{h_\cF}$ is homeomorphic to a space (Section \ref{BAPSsection}) looking as
\begin{equation}\label{BFintro}
\mathscr{B}_\mathscr{F}:=\mathscr{N}_{sing}(X)\bigsqcup (0,1]\times\RR^N,
\end{equation}
where more precisely we prove the following (proposition \ref{BForiented})
\begin{proposition}
The locally compact space $\mathscr{B}_\mathscr{F}$ admits an oriented  $C^\infty$-manifold with boundary structure of dimension $N+1$. 
\end{proposition}

The last proposition is essential to explicitly compute the index
(\ref{Bdiagramintro}) above once the explicit identifications are performed, indeed $\mathscr{B}_\mathscr{F}$ is an oriented cobordism from $\mathscr{N}_{sing}(X)$ to $\RR^N$, we can hence apply a Stoke's theorem argument to obtain the following (proposition \ref{Stokes}):

\begin{proposition}
The index morphism of the deformation space $\mathscr{B}_\cF$ can be computed by means of the following commutative diagram:
$$
\xymatrix{
K^0(\mathscr{N}_{sing}(X))\ar@/^2pc/[rr]^-{ind_{\mathscr{B}_\mathscr{F}}}\ar[rd]_{\int_{\mathscr{N}_{sing}(X)}ch(\cdot)}&
K^0(\mathscr{B}_\mathscr{F})\ar[l]_-{(e_0)_*}^-{\approx}\ar[r]^-{(e_1)_*}& K^0(\mathbb{R}^N)\ar[dl]^{\int_{\RR^N}ch(\cdot)}\\
&\RR&
}
$$
\end{proposition}

Before enouncing the index theorem, we mentioned that there is a Connes-Thom isomorphism and a Morita equivalence
$$K_0(C^*(\mathcal{T}_{nc}X))\stackrel{\mathscr{CT}}{\longrightarrow}K_0(C^*((\mathcal{T}_{nc}X)_{h_0}))\stackrel{Morita}{\longrightarrow}K^0(\mathscr{N}_{sing}(X)).$$
 In Section \ref{CTsection} we develop Connes-Thom using deformation groupoids, this allows us to perform an explicit computation of the above morphism for "noncommutative symbols" $[\sigma_D]\in K_0(C^*(\mathcal{T}_{nc}X))$. 
 
The index theorem is the following:

\begin{theorem}\label{KAPSintro}[K-theoretic APS]
Let $X$ be a manifold with boundary, consider an embedding of $X$ in $\RR^N$ as in \ref{embeddingintro}. The Fredholm index morphism 
$ind_{\cF}:K_0(C^*(\mathcal{T}_{nc}X))\to \mathbb{Z}$ decomposes as the composition of the following three morphisms
\begin{enumerate}
\item A Connes-Thom isomorphism $\mathscr{CT}$:
$$K_0(C^*(\mathcal{T}_{nc}X))\stackrel{\mathscr{CT}}{\longrightarrow}K^0(\mathscr{N}_{sing}(X)),$$
\item The index morphism of the deformation space $\mathscr{B}_\cF$:
$$
\xymatrix{
K^0(\mathscr{N}_{sing}(X))\ar[rr]^-{ind_{\mathscr{B}_\mathscr{F}}}&
& K^0(\mathbb{R}^N)
}
$$
\item the usual Bott periodicity isomorphism:
$$K^0(\mathbb{R}^N)\stackrel{Bott}{\longrightarrow}\mathbb{Z}.$$
\end{enumerate}
In other terms, the following diagram is commutative
$$\xymatrix{
 K_0(C^*(\mathcal{T}_{nc}X)) \ar[d]_-{\mathscr{CT}}^-{\approx} \ar[rrr]^-{ind_{\mathscr{F}}}&&& \mathbb{Z} \\
K^0(\mathscr{N}_{sing}(X)) \ar[rrr]_-{ind_{\mathscr{B}_\mathscr{F}}} &&& K^0(\mathbb{R}^N)\ar[u]_-{Bott}^-{\approx}\\
}$$
\end{theorem}

As discussed above, the three morphisms of the last theorem are computable, and then, exactly as in the classic Atiyah-Singer theorem the last theorem
allows to conclude that, given  an embedding $i:X\hookrightarrow
\RR^N$ as above,  any fully elliptic operator $D$ on $X$ with "non
commutative symbol" $[\sigma_D]\in K_0(C^*(\mathcal{T}_{nc}X))$ gives rise to the following formula:

\vspace{2mm}

\noindent
{\bf Cohomological formula for the APS index (corollary 
\ref{APScohomologyformula})}

\begin{equation}\label{APScohomologyintro}
Index_{APS}(D)=\int_{\mathscr{N}_{sing}(X)}Ch(\mathscr{CT}([\sigma_{D}]))
\end{equation}
where $\int_{\mathscr{N}_{sing}(X)}$ is the integration with respect to the fundamental class of $\mathscr{N}_{sing}(X)$. In Section \ref{APScohomology} we perform an explicit description for $\mathscr{CT}([\sigma_{D}])\in K^0(\mathscr{N}_{sing}(X))$.

\vspace{1mm}

The manifold $\mathscr{N}_{sing}(X)$ (see (\ref{Nsingintro}) above) already reflects an interior contribution and a boundary contribution. In particular, picking up a differential form $\omega_D$ on $\mathscr{N}_{sing}(X)$ representing $Ch(\mathscr{CT}([\sigma_{D}])$, we obtain:
\begin{equation}\label{splitAPSintro}
Index_{APS}(D)=\int_{\mathscr{N}(X) }\omega_D + \, \int_{D_{\partial}} \omega_D.
\end{equation}
The first integral above involves the restriction of $\omega_D$ to $\mathscr{N}(X)$, which is related to the ordinary principal symbol of $D$. More precisely, the principal symbol $\sigma_{pr}(D)$ of $D$ provides a $K$-theory class of $C^*(A^*(\Gamma(X)))$, that is a compactly supported $K$-theory class of the dual of the Lie algebroid of $\Gamma(X)$ or in other words of the $b$-cotangent bundle ${}^bT^*X$, and by functoriality of both the Chern character and Thom-Connes maps, we have 
 $$ [(\omega_D)|_{\mathscr{N}(X)}]= Ch(\mathscr{CT}([\sigma_{pr}(D)]).$$ 
The second integral can thus be viewed as a correction term, which contains the eta invariant appearing in APS formula and which also depends on the choice of the representative $\omega_D\in Ch(\mathscr{CT}([\sigma_{D}]))$.

\vspace{1mm}

{\bf Further developments:} The same methods as above can be applied for manifolds with corners for which we already count with the appropriate b-groupoids (\cite{Mont}) and the appropriate notion of ellipticity as well, which yields criteria for Fredholmness of operators. The generalization of the formula is not however immediate. Indeed, we need to explicitly compute the orbit spaces.
 
Our approach and results are not only a groupoid interpretation of the Atiyah-Patodi-Singer formula. A serious comparison between both formulas has to be done. For instance, as we mentioned above, there is a relation between the second integral on (\ref{splitAPSintro}) and the so called eta invariant. For deeply understanding this, we need to explicitly describe the Chern character of the $\mathscr{CT}([\sigma_{D}]\in K^0(\mathscr{N}_{sing}(X))$, for which one might need to use the Chern character computations done mainly by Bismut in \cite{Bisinv}. Also, the second integral comes from the part of the $b$-groupoid corresponding to the boundary,
$$\partial X\times \partial X\times \RR\times (0,1),$$
and this groupoid's algebra is related with the suspended algebra of Melrose (\cite{Meleta}), a relation between this integral and the Melrose and Melrose-Nistor (\cite{MelNis}), becomes then very interesting to study.  

\section{Groupoids}
\subsection{Preliminaries}

Let us recall some preliminaries on groupoids:

\begin{definition}
A $\it{groupoid}$ consists of the following data:
two sets $\gr$ and $\go$, and maps
\begin{itemize}
\item[(1)]  $s,r:\gr \rightarrow \go$ 
called the source and range (or target) map respectively,
\item[(2)]  $m:\gr^{(2)}\rightarrow \gr$ called the product map 
(where $\gr^{(2)}=\{ (\gamma,\eta)\in \gr \times \gr : s(\gamma)=r(\eta)\}$),
\end{itemize}
such that there exist two maps, $u:\go \rightarrow \gr$ (the unit map) and 
$i:\gr \rightarrow \gr$ (the inverse map),
which, if we denote $m(\gamma,\eta)=\gamma \cdot \eta$, $u(x)=x$ and 
$i(\gamma)=\gamma^{-1}$, satisfy the following properties: 
\begin{itemize}
\item[(i).]$r(\gamma \cdot \eta) =r(\gamma)$ and $s(\gamma \cdot \eta) =s(\eta)$.
\item[(ii).]$\gamma \cdot (\eta \cdot \delta)=(\gamma \cdot \eta )\cdot \delta$, 
$\forall \gamma,\eta,\delta \in \gr$ when this is possible.
\item[(iii).]$\gamma \cdot x = \gamma$ and $x\cdot \eta =\eta$, $\forall
  \gamma,\eta \in \gr$ with $s(\gamma)=x$ and $r(\eta)=x$.
\item[(iv).]$\gamma \cdot \gamma^{-1} =u(r(\gamma))$ and 
$\gamma^{-1} \cdot \gamma =u(s(\gamma))$, $\forall \gamma \in \gr$.
\end{itemize}
Generally, we denote a groupoid by $\gr \rightrightarrows \go $. A morphism $f$ from
a  groupoid   $\hr \rightrightarrows \ho $  to a groupoid   $\gr \rightrightarrows \go $ is  given
by a map $f$ from $\gr$ to $\hr$ which preserves the groupoid structure, i.e.  $f$ commutes with the source, target, unit, inverse  maps, and respects the groupoid product  in the sense that $f(h_1\cdot h_2) = f (h_1) \cdot f(h_2)$ for any $(h_1, h_2) \in \hr^{(2)}$.

\end{definition}

For $A,B$ subsets of $\go$ we use the notation
$\gr_{A}^{B}$ for the subset 
\[
\{ \gamma \in \gr : s(\gamma) \in A,\, 
r(\gamma)\in B\} .
\]

We will also need the following definition:

\begin{definition}[Saturated subgroupoids]\label{defsaturated}
Let $\gr\rightrightarrows M$ be a groupoid.
\begin{enumerate}
\item A subset $A\subset M$ of the units is said to be saturated by $\gr$ (or only saturated if the context is clear enough) if it is union of orbits of $\gr$.
\item A subgroupoid 
\begin{equation}
\xymatrix{
\gr_1 \ar@<.5ex>[d]_{r\ } \ar@<-.5ex>[d]^{\ s}&\subset&\gr \ar@<.5ex>[d]_{r\ } \ar@<-.5ex>[d]^{\ s}  \\
M_1&\subset&M
}
\end{equation}
is a saturated subgroupoid if its set of units $M_1\subset M$ is saturated by $\gr$.

\end{enumerate}
\end{definition}

A groupoid can be endowed with a structure of topological space, or
manifold, for instance. In the case when $\gr$ and $\go$ are smooth
manifolds, and $s,r,m,u$ are smooth maps (with $s$ and $r$
submmersions), then $\gr$ is a Lie groupoid. In the case of manifolds
with boundary, or with corners, this notion can be generalized to that
of continuous families groupoids (see \cite{Pat}).

 
A strict morphism of locally compact groupoids is a groupoid morphism which is continuous. Locally compact groupoids form a category with  strict  morphisms of groupoids. It is now classical in groupoids theory that the right category to consider is the one in which Morita equivalences correspond precisely to isomorphisms.
For more details about the assertions about generalized morphisms written in this Section, the reader can read \cite{TXL} Section 2.1, or 
\cite{HS,Mr,MM}.

We need to introduce some basic definitions, classical when dealing with principal bundles for groups over spaces. 
  
 We recall first the notion of  groupoid action. Given a l.c. groupoid $\gr \rightrightarrows \go$,  
a right  $\gr$-bundle over a manifold $M$ is  a manifold $P$ such
that:
\begin{itemize}
\item $P$ is endowed with maps  $\pi$ and $q$  as in
\[
\xymatrix{
 P \ar[d]^\pi   \ar[rd]^q&\gr \ar@<.5ex>[d]_{r\ } \ar@<-.5ex>[d]^{\ s}  \\
M &\go
}
\]
\item $P$ is endowed with  a continuous right action $\mu: P
  \times_{(q, r )} \gr \to P$,   such that if we denote  $\mu(p,
  \gamma) = p \gamma$, one has $\pi( p \gamma)  = \pi(p)$ and 
$p( \gamma_1\cdot \gamma_2) =   (p  \gamma_1)  \gamma_2 $ for any 
$(\gamma_1, \gamma_2) \in \gr^{(2)}$.  Here  $P \times_{(q, r )} \gr$  denotes  the fiber
product of $q:P\to \go$ and $r: \gr\to \go$. 
\end{itemize}

 A  $\gr$-bundle $P$ is called 
principal if
\begin{enumerate}
\item[(i)] $\pi$ is a surjective submersion, and
\item[(ii)]  the map  $P\times_{(q, r)} \gr \to P\times_M P$,  $(p, \gamma) \mapsto (p,  p\gamma)$  
is a homeomorphism. 
\end{enumerate}

We can now define generalized morphisms between two Lie groupoids.

\begin{definition}[Generalized morphism]\label{HSmorphisms}
Let $\gr \rightrightarrows \go$ and 
$\hr \rightrightarrows \ho$ be two Lie groupoids. 
A generalized morphism (or Hilsum-Skandalis morphism) from $\gr$ to $\hr$,   
$f:  \xymatrix{\hr \ar@{-->}[r] &  \gr}$, is given by the isomorphism class of a right $\gr-$principal bundle over $\hr$, that is, the isomorphism class of:
\begin{itemize}
\item A right  principal $\gr$-bundle $P_f$  over $\ho$ which is also a left $\hr$-bundle
such that the  two  actions commute, formally denoted by
\[
\xymatrix{
\hr \ar@<.5ex>[d]\ar@<-.5ex>[d]&P_f \ar@{->>}[ld] \ar[rd]&\gr \ar@<.5ex>[d]\ar@<-.5ex>[d]\\
\ho&&\go,
}
\]

\end{itemize}
\end{definition}

First of all, usual morphisms between groupoids are also generalized morphisms. Next, as the word suggests it, generalized morphisms can be composed. Indeed, if $P$ and $P'$ are generalized morphisms from $\hr$ to $\gr$ and from $\gr$ to $\lr$ respectively, then 
$$P\times_{\gr}P':=P\times_{\go}P'/(p,p')\sim (p\cdot \gamma, \gamma^{-1}\cdot p')$$
is a generalized morphism from $\hr$ to $\lr$. The composition is associative and thus we can consider the category $Grpd_{HS}$ with objects l.c. groupoids and morphisms given by generalized morphisms. There is a functor
\begin{equation}\label{grpdhs}
Grpd \longrightarrow Grpd_{HS}
\end{equation}
where $Grpd$ is the  category of groupoids with usual morphisms. 

\begin{definition}[Morita equivalence]
Two groupoids are Morita equivalent if they are isomorphic in $Grpd_{HS}$.
\end{definition}


\subsection{Free proper groupoids}

\begin{definition}[Free proper groupoid]
Let $\hr \rightrightarrows \ho$ be a locally compact groupoid. We will say that it is free and proper if it has trivial isotropy groups and it is proper.
\end{definition}

Given a groupoid $\hr\rightrightarrows \ho$, its orbit space is $\mathscr{O}(\hr):=\ho/\sim$, where 
$x\sim y$ iff there is $\gamma\in \hr$ such that $s(\gamma)=x$ and $r(\gamma)=y$.

The following fact is well known. In particular in can be deduced from propositions 2.11, 2.12 and 2.29 in \cite{Tu04}.

\begin{proposition}\label{OMorita}
If $\hr\rightrightarrows\ho$ is a free proper (Lie) groupoid, then
$\hr$ is Morita equivalent  to the locally compact space (manifold) $\mathscr{O}(\hr)$.
\end{proposition}

In fact the Morita bibundle which gives the Morita equivalence is the unit space $\ho$: 
\begin{equation}
\xymatrix{
\hr \ar@<.5ex>[d]\ar@<-.5ex>[d]&\ho \ar@{->>}[ld]_-{Id} \ar[rd]^-{\pi}&\mathscr{O}(\hr) \ar@<.5ex>[d]\ar@<-.5ex>[d]\\
\ho&&\mathscr{O}(\hr),
}
\end{equation}

It is obvious that $\hr$ acts on its units freely and properly if $\hr$ is free and proper, and for the same reason $\mathscr{O}(\hr)$ is a nice locally compact space (even a manifold if the groupoid is Lie).

In particular there is an invertible Hilsum-Skandalis generalized morphism 
\begin{equation}
\mathscr{O}(\hr)--->\hr,
\end{equation}
that can also be given as a 1-cocycle from $\mathscr{O}(\hr)$ with values in $\hr$.
This point of view will be very useful for us in the sequel.

\subsection{Semi-direct groupoids by homomorphisms on $\RR^N$ and Connes-Thom isomorphism}\label{CTsection}

Let $\gr\rightrightarrows M$ be a locally compact groupoid.

We consider $\RR^N$ as an additive group, hence a one unit groupoid. Suppose we have an homomorphism of groupoids
\begin{equation}
\gr \stackrel{h}{\longrightarrow}\RR^N.
\end{equation}

This gives rise to an action (say, a right one) of $\gr$ on the space $M\times\RR^N$  and thus to a new semi-direct groupoid:
\begin{equation}
\gr_h:= (M\times\RR^N)\rtimes  \gr :\gr \times \RR^N\rightrightarrows M\times \RR^N
\end{equation}
 which has the following structural maps:


\begin{itemize}
\item The source and target maps are given by 
\begin{center}
$s(\gamma,X)=(s(\gamma),X+h(\gamma))$ and $r(\gamma,X)=(r(\gamma),X)$
\end{center}
\item The multiplication is defined on composable arrows by the formula
$$(\gamma,X)\cdot(\eta,X+h(\gamma)):=(\gamma\cdot\eta,X).$$
\end{itemize}
Then it is obviously a groupoid with unit map $u(m,X)=(m,X)$ ($h(m)=0$ since $h$ is an homomorphism), and inverse given by $(\gamma,X)^{-1}=(\gamma^{-1},X+h(\gamma))$ (again since we have a homomorphism, $h(\gamma)+h(\gamma^{-1})=0$).

\begin{remark}
For the trivial homomorphism $h_0=0$, the associated groupoid is just the product groupoid
$$\gr\times \RR^N\rightrightarrows M\times \RR^N.$$
\end{remark}

If $\gr$ is provided with a Haar system, then the semi-direct groupoid $\gr_h$ inherits a natural Haar system such that the $C^*$-algebra $C^*(\gr_h)$ is isomorphic to the crossed product algebra 
$C^*(\gr)\rtimes_h \mathbb{R}^N$ where $\mathbb{R}^N$ acts on $C^*(\gr)$ by automorphisms by the formula: 
$\alpha_X(f)(\gamma)=e^{i\cdot (X\cdot h(\gamma))}f(\gamma)$, $\forall
f\in C_c(\gr)$, (see \cite{Concg}, propostion II.5.7 for details). In particular, in the case $N$ is even, we have a Connes-Thom isomorphism in K-theory (\cite{Concg}, II.C)
\begin{equation}\label{CT}
\xymatrix{
K_0(C^*(\gr))\ar[r]^-{\mathscr{CT}}_-{\approx}&K_0(C^*(\gr_h))
}
\end{equation}
which generalizes the classical Thom isomorphism, and which is natural with respect to morphisms of algebras.

Since we will need to  compute explicitly the morphism induced by the
homomorphism we propose an alternative construction of Connes-Thom
which can be computed in our context. More precisely we want to work directly with the groupoid algebras $C^*(\gr_h)$ without passing through the isomorphism with $C^*(\gr)\rtimes_h \mathbb{R}^N$.

Given the morphism $\gr \stackrel{h}{\longrightarrow}\RR^N$ we consider the product groupoid $\gr\times [0,1]\rightrightarrows M\times [0,1]$  of the groupoid $\gr$ with the {\sl space} $[0,1]$ and we define
$$H:\gr\times [0,1]\longrightarrow \RR^N,$$
the homomorphism given by 
$$H(\gamma,t):=t\cdot h(\gamma).$$

This homomorphism gives a deformation between the trivial homomorphism and  $h$, more precisely we have:
\begin{lemma}
Denote by $\gr_H:=(\gr\times [0,1])_H$ the semi-direct groupoid associated to the homomorphism
$H$. For each $t\in [0,1]$ we denote by $(\gr_H)_t$ the restriction subgroupoid $\gr_H|_{M\times \{t\} \times \RR^N}$. We have the following properties:
\begin{enumerate}
\item $(\gr_H)_0=\gr\times \RR^N$
\item $(\gr_H)_1=\gr_h$
\item $(\gr_H)|_{(0,1]}\approx \gr_h\times (0,1]$ and in particular 
$C^*((\gr_H)|_{(0,1]})\approx C^*(\gr_h\times (0,1])$ is contractible.
\end{enumerate}
\end{lemma}
\begin{proof}
The first and second properties are obvious. For the third one can write the explicit isomorphism
\begin{equation}
\xymatrix{
\gr_h\times (0,1] \ar@<.5ex>[d]\ar@<-.5ex>[d]\ar[rr]^-\theta&& (\gr_H)|_{(0,1]}\ar@<.5ex>[d]\ar@<-.5ex>[d]\\
M\times \RR^N\times (0,1]\ar[rr]_{\theta_0}&&M\times (0,1]\times \RR^N,
}
\end{equation}
given by $\theta(\gamma,X,\epsilon)=(\gamma,\epsilon,\epsilon\cdot X)$ and $\theta_0(x,X,\epsilon)=(x,\epsilon,\epsilon\cdot X)$ (the fact that it is a continuous families groupoid morphism is immediate) with explicit inverse given by $\theta^{-1}(\gamma,\epsilon,X)=(\gamma,\frac{X}{\epsilon},\epsilon)$ and $\theta_{0}^{-1}(x,\epsilon,X)=(x, \frac{X}{\epsilon},\epsilon)$
\end{proof}

The last lemma gives rise to a short exact sequence of $C^*$-algebras (\cite{HS,Ren}):
\begin{equation}
0\to C^*(\gr_h\times (0,1])\to C^*(\gr_{H})\stackrel{e_0}{\rightarrow} C^*(\gr\times \RR^N)\to 0,
\end{equation}
where $e_0$ is induced by the evaluation at zero. This defines a deformation index morphism
\begin{equation}
\mathscr{D}_h:K_*(C^*(\gr\times \RR^N))\to K_*(C^*(\gr_h)).
\end{equation}
The natural map $\gr_{H}\to [0,1]$  gives to $\gr_{H}$ the structure of a continuous field of groupoids over $[0,1]$ and if $\gr$ is assumed to be amenable, we get by \cite{LR} that $C^*(\gr_{H})$ is the space of continuous sections of a continuous field of $C^*$-algebras. Then, the deformation index morphism above coincides with the morphism of theorem 3.1 in \cite{ENN93}.

\begin{definition}\label{connes-thom-map}
Let $\gr$ be a groupoid together with a homomorphism $h$ from $\gr$ to $\RR^N$
(with $N$ even).
Consider the morphism in K-theory 
\begin{equation}
K_*(C^*(\gr))\stackrel{\mathscr{CT}_{h}}{\longrightarrow}K_{*}(C^*(\gr_h)),
\end{equation}
given by the composition of the Bott morphism
$$K_*(C^*(\gr))\stackrel{B}{\longrightarrow}K_{*}(C^*(\gr\times \RR^N)),$$
and the deformation index morphism
$$K_{*}(C^*(\gr\times \RR^N))\stackrel{\mathscr{D}_h}{\longrightarrow}K_{*}(C^*(\gr_h)).$$
We will refer to this morphism as the Connes-Thom map associated to $h$.
\end{definition}
In fact, Elliot, Natsume and Nest proved that this morphism coincides with the usual Connes-Thom isomorphism, theorem 5.1 in \cite{ENN93}. We can restate their result in our framework as follows:

\begin{proposition}[Elliot-Natsume-Nest]
Let $(\gr,h)$ be an  amenable continuous family groupoid (or amenable locally compact groupoid with a continuous Haar system) together with a homomorphism on $\RR^N$ ($N$ even). Then the morphism 
$\mathscr{CT}_h:K_*(C^*(\gr))\to K_*(C^*(\gr_h))$ coincides with the Connes-Thom isomorphism. In particular, it satisfies the following properties:
\begin{enumerate}
\item Naturality.
\item If $\gr$ is a space (the groupoid equals its set of units), then $\mathscr{CT}_h$ is the Bott morphism.
\end{enumerate}
\end{proposition}

\section{Noncommutative spaces for manifolds with boundary}

\subsection{The noncommutative tangent space of a manifold with boundary}\label{nctspacesection}
Let $X$ be a manifold with boundary. We denote, as usual, $\intx$ the interior which is a smooth manifold and $\partial X$ its boundary. Let 
\begin{equation}
\Gamma(X)\rightrightarrows X
\end{equation}
be the groupoid of the b-calculus
(\cite{MP,LMNpdo,Mont}). This groupoid has a
pseudodifferential calculus which essentially coincides with Melrose's
$b$-calculus. There is a canonical definition, but in our case we need
to choose a defining function of the boundary, making the definition
simpler. A defining function of the boundary is a smooth function $\rho:X \to
\RR_+$ which is zero on the boundary and only there, and whose
differential is non zero on the boundary. 

\begin{definition}
  The $b$-groupoid of $X$ is 
$$\Gamma(X)=\{ (x,y,\alpha)\in X\times X\times \RR,
\rho(x)=e^\alpha \rho(y)\}.$$
\end{definition}
this implies that 

$$\Gamma(X)=\intx\times\intx \bigsqcup \partial X\times \partial X\times \RR\rightrightarrows X,$$
with groupoid structure compatible with those of $\intx\times\intx$
and $\partial X\times \partial X\times \RR$ ($\RR$ as an additive
group). 
It is a continuous family groupoid, see \cite{Mont,LMNpdo}
for  details. For instance,
$$(x_n,y_n)\to(x,y,\alpha)$$
if and only if
$$x_n\to x,\,y_n\to y \text{ and } log(\frac{\rho(x_n)}{\rho(y_n)})\to \alpha.$$

For such a groupoid it is possible to construct an algebra of
pseudodifferential operators. Although we do not need it in this
article, we recall this background to help the reader relate our work
to usual index theory. See \cite{MP, NWX, Mont, LMNpdo,VassoutJFA} for a
detailed presentation of pseudodifferential calculus on groupoids.

A pseudodifferential operator on a Lie groupoid (or more generally a
continuous family groupoid) $\gr$ is a family of peudodifferential
operators on the fibers of $\gr$ (which are smooth manifolds without
boundary), the family being equivariant under the natural action of $\gr$. 

Compactly supported pseudodifferential operators  form an algebra, denoted by
$\Psi^\infty(\gr)$. The algebra of order 0 pseudodifferential operators
can be completed into a $C^*$-algebra, $\overline{\Psi^0}(\gr)$. There
exists a symbol map, $\sigma$, whose kernel is $C^*(\gr)$. This gives
rise to the following exact sequence:
$$0 \to C^*(\gr) \to \overline{\Psi^0}(\gr) \to C_0(S^*(\gr))$$
where $S^*(\gr)$ is the cosphere bundle of the Lie algebroid of $\gr$.

In the general context of index theory on groupoids, there is an
analytic index which can be defined in two ways. The first way,
classical, is to consider the boundary map of the 6-terms exact
sequence in $K$-theory induced by the short exact sequence above:
$$ind_a: K_1(C_0(S^*(\gr))) \to K_0(C^*(\gr)).$$

Actually, an alternative is to define it through the tangent groupoid
of Connes, which was originally defined for the groupoid of a smooth
manifold and later extended to the case of continuous family groupoids
(\cite{MP,LMNpdo}). In general, 
$$\gr^{tan}=A(\gr)\bigsqcup \gr\times (0,1]\rightrightarrows \gr^{(0)}\times
[0,1],$$ 
with the deformation to the normal cone structure, see definition \ref{defgtan}. In particular the Lie algebroid can be identified with 
$A\gr=Ker\,ds|_{\gr^{(0)}} $, see (\ref{AGds}) for more details.

Using the evaluation maps, one has two $K$-theory morphisms, $e_0: K_*(C^*(\gr^{tan}))
\to K^*(A\gr)$ which is an isomorphism (since $K_*(C^*(\gr\times (0,1]))=0$), and $e_1: K_*(C^*(\gr^{tan})) \to K_*(C^*(\gr))$. The analytic
index can be defined as
$$ind_a=e_1 \circ e_0^{-1}:K^*(A\gr) \to K_*(C^*(\gr)).$$
It is thus possible to work on index problems without using the
algebra of pseudodifferential operators. In the rest of this article,
we will only use deformation groupoids like the tangent groupoid, and
not pseudodifferential algebras.

But in general, this analytic index is not the Fredholm index. In
certain cases, it is possible to define the latter using a refinement
of the tangent groupoid.

In order to use explicitly the tangent groupoid in our case, a little discussion on the algebroid is needed. For this, remark that we can define the vector bundle $TX$ over $X$ as the restriction of the tangent space of a smooth manifold $\tilde{X}$, for example we will use the double of $X$ (gluing along the boundary with the defining function $\rho$), on which $X$ is included. So, $TX:=T_X\tilde{X}$ and we can use the defining function to trivialize the normal bundle of $\partial X$ in $\tilde{X}$ to identify $T_X\tilde{X}$ with the bundle with fibers 
\begin{center}
$T_X\tilde{X}_x=T_x\intx$ if $x\in \intx$ and $T_X\tilde{X}_x=T_x\partial X \times \RR$ if $x\in \partial X$.
\end{center}



Now, $\Gamma(X)$ is a continuous family groupoid and as such one can define $T(\Gamma(X))$, see \cite{MP,LMNpdo,Pat2000} for more details. What it is important for us is its restriction to $X$, where $X\hookrightarrow \Gamma(X)$ as the groupoid's units. This vector bundle $T_X(\Gamma(X))$ over $X$ has fibers
\begin{center}
$T_X(\Gamma(X))_x=T_{(x,x)}(\intx\times\intx)$ if $x\in \intx$ and 
\end{center}
\begin{center}
$T_X\Gamma(X)_x=T_{(x,x)}(\partial X \times \partial X) \times\RR\times\RR$ if $x\in \partial X$.
\end{center}
The algebroid $A(\Gamma(X))\to X$ is defined as the restriction to the unit space $X$ of the vector bundle $T\Gamma(X)=\cup_{x\in X} T\Gamma(X)_{x}$ over the groupoid $\Gamma(X)$, in particular:
\begin{center}
$A(\Gamma(X))_x\simeq T_x\intx$ if $x\in \intx$ and $A(\Gamma(X))_x\simeq T_x\partial X \times \RR$ if $x\in \partial X$.
\end{center}

Take the tangent groupoid associated to $\Gamma(X)$:
$$\Gamma(X)^{tan}=A(\Gamma(X))\bigsqcup \Gamma(X)\times (0,1]\rightrightarrows X\times [0,1]$$

Let us now consider the open subgroupoid of $\Gamma(X)^{tan}$ given by the restriction to 
$X_\cF :=X\times [0,1]\setminus (\partial X\times \{1\})$ :

\begin{equation}\nonumber
\Gamma(X)_{\cF}:=A(\Gamma(X))\bigsqcup \partial X\times \partial X\times \RR\times (0,1)\bigsqcup \intx\times \intx \times (0,1]\rightrightarrows X_\cF.
\end{equation}

As we will discuss below, the deformation index morphism associated to this groupoid gives precisely the Fredholm index. But let us continue with our construction of noncommutative spaces.

By definition $$\intx\times \intx \times (0,1]\subset \Gamma(X)_\cF$$ as a saturated open dense subgroupoid. We can then obtain a complementary closed subgroupoid of $\Gamma(X)_\cF$ :
\begin{equation}
\mathcal{T}_{nc}X\rightrightarrows X_{\partial},
\end{equation}
where $X_{\partial}:=X_\cF\setminus \intx \times (0,1]$.

 To be more descriptive, the groupoid looks like
$$\mathcal{T}_{nc}X=A(\Gamma(X))\bigsqcup \partial X\times \partial X\times \RR\times (0,1).$$

\begin{definition}
The groupoid $\mathcal{T}_{nc}X$ will be called {\it "The noncommutative tangent space of $X$"}. We will also refer to $\Gamma(X)_{\cF}$ as {\it "The Fredholm tangent groupoid of X"}. 
\end{definition}

\subsection{Embeddings of manifolds with boundary and proper free groupoids}\label{sectionembedding}

To define a homomorphism
$\Gamma(X)_{\cF}\stackrel{h}{\longrightarrow}\mathbb{R}^N$ we will
need as in the nonboundary case an appropiate embedding. We recall the
construction outlined in \cite{CMont}.
Consider an embedding
$$i_{\partial}:X\hookrightarrow \mathbb{R}^{N-1},$$
in which $N$ is an even integer.
The embedding we are going to use is
\begin{equation}\label{embedding}
i:X\hookrightarrow \mathbb{R}^{N-1}\times \mathbb{R}_+
\end{equation}
given by 
$$i(x)=(i_{\partial}(x),\rho(x)),$$
where $\rho:X\to \mathbb{R}_+$ is a defining function of the boundary.

We can then use it to define a homomorphism 
$h:\Gamma(X)\rightarrow \mathbb{R}^{N}$ as
follows.
\begin{equation}\label{hdefinition}
  h:
  \begin{cases}
    h(x,y)=(i_{\partial}(x)-i_{\partial}(y),\,log(\frac{\rho(x)}{\rho(y)}))  \text{ on }
    \intx \times\intx\\
h(x,y,\alpha)=(i_{\partial}(x)-i_{\partial}(y),\alpha) \text{ on } \partial X\times \partial X \times
\mathbb{R} \\
  \end{cases}
\end{equation}

By the definition of the topology on $\Gamma(X)$, it is clearly a continuous family groupoid morphism. 

The interest of defining a good morphism $h$ is that the induced
groupoid will be free and proper, like in the case of a smooth manifold, as it was shown by Connes. The freeness follows from a general principle: if we have a homomorphism $\gr\stackrel{h}{\longrightarrow}\RR^N$ then the isotropy groups of the induced groupoid $\gr_h$ are 
$$(\gr_h)_{(x,X)}^{(x,X)}=\{(\gamma,X)\in \gr\times \RR^N/\gamma\in \gr_x^x, \, h(\gamma)=0\},$$
where as usual for a groupoid $\gr\rightrightarrows \gr_0$, $\gr_x^x:=s^{-1}(\{x\})\cap t^{-1}(\{x\})$ for a unit $x\in \gr_0$ denotes the isotropy group of $x$.
In particular, if $h$ is a monomorphism ($h(\gamma)=0$ iff $\gamma$ is a unit), it implies that the isotropy groups described above are trivial, $(\gr_h)_{(x,X)}^{(x,X)}=\{(x,X)\}$\footnote{As is usual for groupoids, we identify the space of units with the space of identity arrows}.

We will now check the properness of the groupoids $\gr_h$ we will be using. For that we will use the general properness condition $(ii)$ of proposition 2.14 in \cite{Tu04} which tell us that we need to check two things: 
\begin{itemize} 
\item[(A)] The map $$\gr_h\stackrel{(r,s)}{\longrightarrow}(\gr^{(0)}\times \RR^N)\times (\gr^{(0)}\times \RR^N)$$ is closed, and 

\item[(B)] For every $(a,X)\in \gr^{(0)}\times \RR^N$ the stabilizers $(\gr_h)_{(a,X)}:=\{\gamma\in \gr:t(\gamma)=a=s(\gamma)\, and \, X=X+h(\gamma) \}$ are quasi-compact. 
\end{itemize} 

In our case property (B) is immediately verified, indeed, as we mentioned above, $h$ is a monomorphism so the stabilizers are trivial, $(\Gamma(X)_\mathscr{F})_{(a,X)}=\{(a,X)\}$, hence quasi-compact.

\begin{lemma}
The induced semi-direct groupoid $\Gamma(X)_h$ is a free proper groupoid.
\end{lemma}

\begin{proof}
As mentioned above we have to verify only property (A), that is, we have to check that the map
$$\Gamma(X)\times \RR^N\stackrel{(r,s)}{\longrightarrow} (X\times \RR^N)\times (X\times \RR^N)$$
given by
\begin{equation}
 (r,s):
  \begin{cases}
    (x,y,X)\mapsto ((x,X),(y,X+(i_{\partial}(x)-i_{\partial}(y),\,log(\frac{\rho(x)}{\rho(y)})))  \text{ on }
    \intx \times\intx \times \RR^N\\
(x,y,\alpha,X)\mapsto((x,X),(y,X+(i_{\partial}(x)-i_{\partial}(y),\alpha) )\text{ on } \partial X\times \partial X \times
\mathbb{R} \times \RR^N\\
  \end{cases}
\end{equation}
is closed.

Let $(A_n)_n:=(\gamma_n,X_n)_n$ a sequence in $\Gamma(X)\times \RR^N$ such that 
\begin{equation}\label{convergencegamma1}
lim_{n\to \infty}(r,s)(\gamma_n,X_n)=P
\end{equation}
with $P$ a point in $(X\times \RR^N)\times (X\times \RR^N)$. It is enough to justify that there is a subsequence of $(A_n)_n$ converging to an antecedent of $P$: We will separate the analysis in four cases
\begin{itemize}
\item[$(a)$] Suppose $P=((x,X),(y,Y))$ with $x\in \intx$ and $y\in \partial X$. Since $\intx$ is open in $X$ we have that $\gamma_n\in \intx\times \intx$ from a certain large enough $n$, that is, we might suppose that for each $n$, $\gamma_n=(x_n,y_n)$ for some $x_n,y_n\in \intx$ and in particular 
$$(r,s)(\gamma_n,X_n)=((x_n,X_n),(y_n,X_n+(i_{\partial}(x_n)-i_{\partial}(y_n),\,log(\frac{\rho(x_n)}{\rho(y_n)})))).$$
Now, the limit (\ref{convergencegamma1}) above implies the following convergences:
$x_n\to x$, $y_n\to y$, $X_n\to X$ and $X_n+(i_{\partial}(x_n)-i_{\partial}(y_n),\,log(\frac{\rho(x_n)}{\rho(y_n)}))\to Y$. Hence we also have that 
$log(\frac{\rho(x_n)}{\rho(y_n)})$ converges, but this is impossible since 
$\rho(x_n)\to \rho(x)> 0$ and $\rho(y_n)\to \rho(y)= 0$. This case is thus not possible.

\item[$(b)$] Suppose $P=((x,X),(y,Y))$ with $x\in \partial X$ and $y\in \intx$. This case is symmetric to the precedent one, the same analysis shows is empty. 

\item[$(c)$] Suppose $P=((x,X),(y,Y))$ with $x\in \intx$ and $y\in \intx$. We might suppose again that for each $n$, $\gamma_n=(x_n,y_n)$ for some $x_n,y_n\in \intx$ and in particular 
$$(r,s)(\gamma_n,X_n)=((x_n,X_n),(y_n,X_n+(i_{\partial}(x_n)-i_{\partial}(y_n),\,log(\frac{\rho(x_n)}{\rho(y_n)})))).$$
Let $A=(x,y, X)$, $A\in \Gamma(X)\times \RR^N$ and $(r,s)(A)=P$. The limit (\ref{convergencegamma1}) implies that we have the convergence $A_n\to A$.

\item[$(d)$] Finally, suppose $P=((x,X),(y,Y))$ with $x\in \partial X$ and $y\in \partial X$. In this case we have two possibilities\footnote{the two options may coexist.}, $(d_1):$ either $A_n$ has a subsequence completely contained in $\intx\times\intx\times \RR^N$, or $(d_2):$ 
$A_n$ has a subsequence completely contained in $\partial X\times \partial X\times \RR\times \RR^N$. 

In the case $(d_1)$ we might suppose again that for each $n$, $\gamma_n=(x_n,y_n)$ for some $x_n,y_n\in \intx$. The limit (\ref{convergencegamma1}) above implies the following convergences:
$x_n\to x$, $y_n\to y$, $X_n\to X$ and $X_n+(i_{\partial}(x_n)-i_{\partial}(y_n),\,log(\frac{\rho(x_n)}{\rho(y_n)}))\to Y$. In particular we also have 
$log(\frac{\rho(x_n)}{\rho(y_n)})$ converges to a certain $\alpha\in \RR$.
Hence, letting 
$A=(x,y,\alpha, X)$ we have that $(r,s)(A)=P$ and $A_n$ converges to 
$A$.

In the case $(d_2)$, we might suppose that for each $n$, $\gamma_n=(x_n,y_n,\alpha_n)$ with $x_n,y_n\in \partial X$ and $\alpha_n\in \RR$. The limit (\ref{convergencegamma1}) above implies the following convergences:
$x_n\to x$, $y_n\to y$, $X_n\to X$ and $X_n+(i_{\partial}(x_n)-i_{\partial}(y_n),\alpha_n)\to Y$. Thus $\alpha_n$ converges too to a $\alpha\in \RR$.
Hence, letting 
$A=(x,y,\alpha, X)$ we have that $(r,s)(A)=P$ and $A_n$ converges to 
$A$.

\end{itemize}

\end{proof}
Now, the morphism $h$ induces a morphism between the algebroids (see (\ref{AGmorphism}) for more details), 
$$A(h):A(\Gamma(X))\to A(\RR^N)=\RR^N.$$
With the identification that we have for the algebroid, we explicitly have
\begin{equation} 
A(h)(x,V)=d_xi^+(V) 
\end{equation} 
if $x\in \intx$ and $V\in T_x\intx$, where $i^+:\intx\to\RR^N$ is defined as $i^+(x):=(i_{\partial}(x),log(\rho(x)))$; and 
\begin{equation} 
A(h)((x,\xi),\alpha)=(d_xi_{\partial}(\xi),\alpha) 
\end{equation} 
if $x\in \partial X$, $\xi\in T_x\partial X$ and $\alpha\in \RR$, where $i_{\partial}$ is the restriction of $i_{\partial}$ to $\partial X$.
We also have the properness of the respective action.
\begin{lemma}
The induced semi-direct groupoid $A(\Gamma(X))_{A(h)}$ is a free proper groupoid.
\end{lemma}
\begin{proof}
Again, we have to verify only property (A), that is, we have to check that the map
$$A(\Gamma(X))\times \RR^N\stackrel{(r,s)}{\longrightarrow} (X\times \RR^N)\times (X\times \RR^N)$$
given by
\begin{equation}
 (r,s):
  \begin{cases}
    ((x,V),X)\mapsto ((x,X),(x,X+d_xi^+(V))  \text{ on }
    T\intx \times \RR^N\\
((x,\xi),X)\mapsto((x,X),(x,X+(d_xi_{\partial}(\xi),\alpha) )\text{ on } 
T\partial X \times
\mathbb{R} \times \RR^N\\
  \end{cases}
\end{equation}
is closed.

Let $(A_n)_n:=(\gamma_n,X_n)_n$ a sequence in $A(\Gamma(X))\times \RR^N$ such that 
\begin{equation}\label{convergencegamma}
lim_{n\to \infty}(r,s)(\gamma_n,X_n)=P
\end{equation}
with $P$ a point in $(X\times \RR^N)\times (X\times \RR^N)$. It is enough to justify that there is a subsequence of $(A_n)_n$ converging to an antecedent of $P$: We will separate the analysis in four cases
\begin{itemize}
\item[$(a)$] Suppose $P=((x,X),(y,Y))$ with $x\in \intx$ and $y\in \partial X$. Since $\intx$ is open in $X$ we have that $\gamma_n\in T\intx$ from a certain large enough $n$, that is, we might suppose that for each $n$, $\gamma_n=(x_n,V_n)$ for some $V_n\in T_{x_n}\intx$ and in particular 
$$(r,s)(\gamma_n,X_n)=((x_n,X_n),(x_n,X_n+d_{x_n}i^+(V_n))).$$
Now, the limit (\ref{convergencegamma}) above implies in particular the following convergences:
$x_n\to x$ and $x_n\to y$. This case is thus not possible.

\item[$(b)$] The case $P=((x,X),(y,Y))$ with $x\in \partial X$ and $y\in \intx$ is empty, the argument of $(a)$ above applies as well.

\item[$(c)$] Suppose $P=((x,X),(y,Y))$ with $x\in \intx$ and $y\in \intx$. We might suppose again that for each $n$, $\gamma_n=(x_n,V_n)$ for some $\xi_n\in T_{x_n}\intx$ and in particular 
$$(r,s)(\gamma_n,X_n)=((x_n,X_n),(x_n,X_n+d_{x_n}i^+(V_n))).$$
The limit (\ref{convergencegamma}) implies that $x_n\to x$, $X_n\to X$ and $X_n+d_{x_n}i^+(V_n)\to Y$, in particular $(d_{x_n}i^+(V_n))_n$ converges in $\RR^N$ too. Now, since $i^+$ is an embedding we have that $di^+$ is a closed embedding, in other words there is a $V\in T_x\intx$ such that $d_xi^+(V)$ is the limit of $(d_{x_n}i^+(V_n))_n$. Hence, letting 
$A=((x,V), X)\in T\intx \times \RR^N$ we have that $(r,s)(A)=P$ and $A_n$ converges to $A$.

\item[$(d)$] Finally, suppose $P=((x,X),(y,Y))$ with $x\in \partial X$ and $y\in \partial X$. In this case we have two possibilities, $(d_1):$ either $A_n$ has a subsequence completely contained in $T\intx\times \RR^N$, or $(d_2):$ 
$A_n$ has a subsequence completely contained in $T\partial X\times \RR\times \RR^N$.

In the case $(d_1)$ we might suppose again that for each $n$, $\gamma_n=(x_n,V_n)$ for some $V_n\in T_{x_n}\intx$. The limit (\ref{convergencegamma}) above implies the following convergences:
$x_n\to x$, $X_n\to X$ and $X_n+d_{x_n}i^+(V_n)\to Y$. In particular we also have 
$d_{x_n}i^+(V_n)$ converges in $\RR^N$. Again, there is then a $V\in T_x\intx$ such that $d_xi^+(V)$ is the limit of $(d_{x_n}i^+(V_n))_n$.
Hence, letting 
$A=((x,V), X)$ we have that $(r,s)(A)=P$ and $A_n$ converges to 
$A$.

In the case $(d_2)$, we might suppose that for each $n$, $\gamma_n=((x_n,\xi_n),\alpha_n)$ with $\xi_n\in T_{x_n}\partial X$ and $\alpha_n\in \RR$. The limit (\ref{convergencegamma}) above implies the following convergences:
$x_n\to x$, $X_n\to X$ and $X_n+(d_{x_n}i_{\partial}(\xi_n),\alpha_n)\to Y$. Thus $d_{x_n}i_{\partial}(\xi_n)$ and $\alpha_n$ converge too in $\RR^{N-1}$ and  in $\RR$ respectively.
Again, because $di_{\partial}$ is a closed embedding, there is a $\xi \in T_x\partial X$ such that $d_xi_{\partial}(\xi)$ is the limit of $(d_{x_n}i_{\partial}(\xi_n))_n$. Letting 
$A=((x,\xi),\alpha, X)\in T_\partial X\times \RR\times \RR^N$ with $\alpha$ the limit of $(\alpha_n)_n$, we have that $(r,s)(A)=P$ and $A_n$ converges to $A$.

\end{itemize}

\end{proof}

Let us now apply to $h$ the tangent groupoid functor to obtain a continuous family groupoid morphism
$$h^{tan}: \Gamma(X)^{tan}\to (\RR^N)^{tan},$$
explicitly given by 
\begin{equation}\label{htandefinition}
  h^{tan}:
  \begin{cases}
    h^{tan}(\gamma,\varepsilon)=(h(\gamma),\varepsilon)  \text{ on }
    \Gamma(X)^{tan}\times (0,1]\\
h^{tan}(x,\xi)=A(h)(x,\xi) \text{ on } A(\Gamma(X)) \\
  \end{cases}
\end{equation}

Remember now that the tangent groupoid of $\RR^N$ (as an additive group) is diffeomorphic to $\RR^N\times [0,1]$ by the diffeormorphism $(\RR^N)^{tan}\to \RR^N\times [0,1]$ given by 
\begin{equation}\label{RNtan}
  \begin{cases}
    (X,0)\mapsto (X,0)  \text{ on }
    \RR^N \times \{0\}\\
(X,\varepsilon)\mapsto (\frac{X}{\varepsilon},\varepsilon) \text{ on } \RR^N\times(0,1] \\
  \end{cases}
\end{equation}

As a corollary of proposition \ref{tanproperaction} and the two lemmas above we have
\begin{corollary}
Consider the continuous family groupoids morphism $\Gamma(X)^{tan}\stackrel{h^T}{\to}\RR^N$ given as the composition of $h^{tan}$ composed with the diffeomorphism $(\RR^N)^{tan}\approx \RR^N\times[0,1]$ and finally with the projection on $\RR^N$. 

Then the semi-direct groupoid $(\Gamma(X)^{tan})_{h^T}$ is a free proper groupoid.
\end{corollary}

We consider finally the morphism on the Fredholm groupoid
\begin{equation}
h_{\cF}:\Gamma(X)_{\cF}\to \RR^N
\end{equation} 
given by the restriction of $h^T$ to
$\Gamma(X)_{\cF}$.

We have obtained in particular the following result:

\begin{proposition}\label{hcontinue}
 $h_{\cF}:\Gamma(X)_{\cF}\rightarrow \mathbb{R}^N$ defines a homomorphism of continuous family groupoids and the groupoid $(\Gamma(X)_{\cF})_{h_\cF}$ is a free proper groupoid.
\end{proposition}

\begin{remark}
As an immediate consequence of the proposition above, the groupoid
$(\Gamma(X)_{\cF})_h$ is Morita equivalent to its space of orbits, (see proposition \ref{OMorita}). 
\end{remark}

\subsection*{The index morphism of the orbit space of $(\Gamma(X)_{\cF})_h$}

Let us denote by $\mathscr{B}_h=X_\cF\times \RR^N/\sim_h$ the space of orbits of
$(\Gamma(X)_{\cF})_h$. Remember that to  define the deformation index morphism associated to 
the Fredholm groupoid we considered the saturated open dense subgroupoid 
$\intx\times\intx \times (0,1]\rightrightarrows \intx\times (0,1]$ and
its closed complement $\mathcal{T}_{nc}(X)\rightrightarrows
X_{\partial}$. The restrictions of $h$ to these two subgroupoids have
the same properties as $h$, the induced actions are free and proper. Moreover, since we are dealing with saturated subgroupoids, we have a good behaviour at the level of orbit spaces. That is, denoting by $\mathscr{B}_{h_0}=X_{\partial}\times \RR^N/\sim_{h_0}$ and $\mathscr{B}_{h_1}=\intx\times \RR^N/\sim_{h_1}$ the orbit spaces of $(\mathcal{T}_{nc}(X))_{h_0}$ and $(\intx\times\intx)_{h_1}$ respectively, we have an index morphism

\begin{equation}\label{indexBh}
\xymatrix{
K^0(\mathscr{B}_{h_0})&
K^0(\mathscr{B}_h)\ar[l]_-{(e_0)_*}^-{\approx}\ar[r]^-{(e_1)_*}& K^0(\mathscr{B}_{h_1})
}
\end{equation}
by considering the open dense subset $\mathscr{B}_{h_{(0,1]}}=\intx\times(0,1]\times\RR^N/\sim$ of $\mathscr{B}_h$.

We want next to fully understand this index morphism. Since we are now dealing with spaces this should be possible. 

\begin{lemma}\label{lemmaqint}
We have an homeomorphism between the open dense subset $\mathscr{B}_{h_{(0,1]}}=\intx\times(0,1]\times\RR^N/\sim$ of $\mathscr{B}_h$ and 
$(0,1]\times\RR^N$. More explicitly, the map
$$(x,\varepsilon, X)\mapsto (\varepsilon, \varepsilon\cdot X+i(x))$$
passes to the quotient into an homeomorphism
$$\widetilde{\overset{\:\circ}{q}}:\intx\times(0,1]\times\RR^N/\sim\longrightarrow (0,1]\times\RR^N$$
\end{lemma}

\begin{proof}
The map $$\overset{\:\circ}{q}:(x,\varepsilon, X)\mapsto (\varepsilon, \varepsilon\cdot X+i(x))$$
is obviously a continuous open surjection from $\intx\times(0,1]\times\RR^N$ to $(0,1]\times\RR^N$.
Moreover, by definition 
$$\overset{\:\circ}{q}(x,\varepsilon, X)=\overset{\:\circ}{q}(y,\varepsilon', Y)$$ if and only if $$\varepsilon=\varepsilon'\qquad and \qquad Y=X+\frac{x-y}{\varepsilon}$$
that is, if and only if $$(x,\varepsilon, X)\sim_h (y,\varepsilon', Y).$$
The conclusion follows now immediately.
\end{proof}

\subsubsection{Singular normal bundle}\label{singnbsection}
We will need to describe as well the closed complement of the open subset considered above. 

Let us consider first the normal bundle $N(X)$ over $X$ associated to the embedding $X\stackrel{\iota}{\hookrightarrow} \mathbb{R}^{N-1}\times \mathbb{R}_+$ given in (\ref{embedding}) above\footnote{Notice that $N(X)\to X$ is an honest vector bundle}. Now, we can consider the {\it Singular normal bundle} of $X$ associated to the same embedding (\ref{embedding}) which is the $C^{\infty}$-manifold of dimension $N$ whose subjacent set is
\begin{equation}
\mathscr{N}_{sing}(X):=N(X)\times \{0\}\bigsqcup \RR^{N-1}\times (0,1)
\end{equation}
and whose manifold structure we will now explicitly describe.


The structure is such that $N(X)|_{\intx}$ and $\RR^{N-1}\times (0,1)$ are two open submanifolds. We have then to describe the structure around $N(X)|_{\partial X}$: 

Let $V\subset\RR^{dim\partial X}$ be an open subset, consider 
$$W_V^-:=V\times \RR^{N-1-dim\partial X}\times (-1,0].$$
Take now, $U:=V\times (-1,1)^{N-1-dim\partial X}\subset \RR^{N-1}$ and 
consider 
$$W_V^+:=\{(a,Y,\varepsilon)\in \RR^{dim\partial X}\times\RR^{N-1-dim\partial X}\times [0,1): a+\varepsilon Y\in U\}.$$
In particular remark that $W_V^+\bigcap\RR^{dim\partial X}\times\RR^{N-1-dim\partial X}\times \{0\}=V\times\{0\}\times \{0\}$

We can consider the open subset of $\RR^N$:
$$W_V:=W_V^-\bigcup W_V^+.$$
We define charts 
\begin{equation}\label{chartW}
W_V\stackrel{\Psi}{\longrightarrow}\cW_\cV \subset \mathscr{N}_{sing}(X)
\end{equation}
around $N(X)|_{\partial X}$ by taking charts $V\stackrel{\phi}{\approx}\cV$
covering $\partial X$ and such that we have trivializations
$$V\times \RR^{N-1-dim\partial X}\times (-1,0]\approx \cV\times \RR^{N-1-dim\partial X}\times (-1,0]\approx N(X)|_{\cV\times (-1,0]}$$
Thus obtaining in this way 
\begin{equation}\label{chartW-}W_V^-\stackrel{\Psi^-}{\longrightarrow}\cW_\cV^-:=N(X)|_{\cV\times (-1,0]}.
\end{equation}
For $W_V^+$, we might suppose that $\phi$ gives a slice chart of $\partial X$ in $\RR^{N-1}$ diffeomorphic to $$U:=V\times (-1,1)^{N-1-dim\partial X}.$$ Then the $W_V^+$ are precisely the open subsets $\Omega_V^U$ considered in \cite{Ca2} Section 3. We consider the deformation to the normal cone charts explicitly described in \cite{Ca2} Section 3, proposition 3.1 (see also \cite{Concg} Section II.5).
\begin{equation}\label{chartW+}
W_V^+\stackrel{\Psi^+}{\longrightarrow}\cW_\cV^+\subset D_\partial.
\end{equation}

Locally they look like:

$$\Omega_V^U\to D_V^{U}\approx D_\cV^{\cU} \subset D_\partial$$
$$(a,Y,0)\mapsto (a,Y,0), \, and$$
$$(a,Y,\varepsilon)\mapsto (a+\varepsilon Y,\varepsilon)\, for \, \varepsilon\neq0.$$

The fact that $\{(\cW_\cV,\Psi)\}$ are compatible with $N(X)|_{\intx}$ is immediate and the fact that they are compatible with $\RR^{N-1}\times (0,1)$ follows from proposition 3.1 in \cite{Ca2}.

\begin{lemma}\label{lemmaqbord}
We have an homeomorphism between the closed subset $\mathscr{B}_{h_0}=X_\partial \times\RR^N/\sim$ of $\mathscr{B}_h$ and 
$\mathscr{N}_{sing}(X)$. More explicitly, the map
\begin{equation}\label{qpartial}
\begin{cases}
(x,\varepsilon, X)\mapsto (\varepsilon, \varepsilon\cdot X_{N-1}+i_\partial(x))\\
(x,0,X)\mapsto [X]_x\in N_x(X)
\end{cases}
\end{equation}
passes to the quotient into an homeomorphism
$$\widetilde{q_\partial}:X_\partial\times\RR^N/\sim\longrightarrow \mathscr{N}_{sing}(X)$$
\end{lemma}

\begin{proof}
Let us denote by $q_\partial:X_\partial \times\RR^N\to\mathscr{N}_{sing}(X)$ the map (\ref{qpartial}) above. We will show that it is an open continuous surjection. The fact that is surjective is immediate. For proving that it is open and continuous, it is enough to check that for every point 
$p\in X_\partial \times\RR^N$ there is an open neighborhood $\cV_p$ such that $q_\partial(\cV_p)$ is open and $q_\partial:\cV_p\to q_\partial(\cV_p)$ is continuous.

For points in $\partial X\times (0,1)\times \RR^N$, the restriction
$$\partial X\times (0,1)\times \RR^N\to q_\partial(\partial X\times (0,1)\times \RR^N)=(0,1)\times\RR^{N-1}$$
is given by $(x,\varepsilon, X)\mapsto (\varepsilon, \varepsilon\cdot X_{N-1}+i_\partial(x))$ which is open and continuous (the proof is completely analog to lemma \ref{lemmaqint}).

For points in $\intx\times \{0\}\times \RR^N$, the restriction
$$\intx\times \{0\}\times \RR^N\to q_\partial(\intx\times \{0\}\times \RR^N)=N(X)|_{\intx}$$
is given by $(x,0, X)\mapsto[X]_x\in N_x(X)$ which is clearly open and continuous.

Finally, for points in $\partial X\times \{0\}\times \RR^N$ we need to have more careful. Let $x_0\in \partial X$, we take a chart $V\stackrel{\phi}{\approx}\cV$ as above, together with the correspondant slice chart $U\in \RR^{N-1}$. Around a point $(x,0,X_0)\in X_\mathscr{F}\times \RR^N$ we can consider a chart (with boundary) diffeomorphic to
$$\cU:=\{(x,t,\varepsilon,X)\in V\times (-1,0]\times [0,1)\times \RR^N:
x+\varepsilon X_{N-1-dim\partial X}\in U\}.$$
Take $\cU_\partial$ the corresponding intersection with $X_\partial \times \RR^N$. In these local coordinates, the map $q_\partial$ looks like:
$$\cU\to W_V$$
given by
$$(x,0,\varepsilon, X)\mapsto (x,X_{N-1-dim\partial X},\varepsilon)\in W_V^+$$
and
$$(x,t,0,X)\mapsto(x,X_{N-1-dim\partial X},t)\in W_V^-.$$
It is evidently an open continuous map.
Hence the map $q_\partial$ is an open continuous surjection for which $q_\partial(a)=q_\partial(b)$ if and only if $a\sim_h b$ in $X_\partial \times \RR^N$. We conclude then that the induced map is an homeomorphism.
\end{proof}

\subsubsection{The APS classifying space}\label{BAPSsection}
Consider 
\begin{equation}\label{BF}
\mathscr{B}_\mathscr{F}:=\mathscr{N}_{sing}(X)\bigsqcup (0,1]\times\RR^N.
\end{equation}
By the two precedent lemmas we can conclude that there is an unique locally compact topology on $\mathscr{B}_\mathscr{F}$ such that the bijection 
$$\widetilde{q_\mathscr{F}}:=\widetilde{q_\partial}\bigsqcup \widetilde{\overset{\:\circ}{q}}:\mathscr{B}_h\longrightarrow \mathscr{B}_\mathscr{F}$$
induced by $\widetilde{q_\partial}$ and $\widetilde{\overset{\:\circ}{q}}$ is an homeomorphism.

With the new identifications, the index morphism (\ref{indexBh}) takes the following form:
\begin{equation}\label{indexBF}
\xymatrix{
K^0(\mathscr{N}_{sing}(X))&
K^0(\mathscr{B}_\mathscr{F})\ar[l]_-{(e_0)_*}^-{\approx}\ar[r]^-{(e_1)_*}& K^0(\RR^N).
}
\end{equation}
The APS index theorem proposed below, will be useful only if we can compute this index morphism. In the case $\partial X=\emptyset$ the topology of $\mathscr{B}_\mathscr{F}$ gives immediately that the above index is just the shriek map induced by an open inclusion. In the case with boundary, $\mathscr{N}_{sing}(X)$ is not quite an open subset of $\mathbb{R}^N$, but still we can prove the following by analyzing the topology of $\mathscr{B}_\mathscr{F}$:

To describe the topology of this space we will describe three big open subsets that cover $\mathscr{B}_\mathscr{F}$ and that generate the entire topology (see proposition below). Let us first define the subjacent subsets $U_\mathscr{F}^0$, $U_\mathscr{F}^1$ and $U_\mathscr{F}^2$: 
$$U_\mathscr{F}^0=\RR^N\times (0,1],$$
$$U_\mathscr{F}^1=(\RR^{N-1}\times (0,1))\bigsqcup (\RR^{N-1}\times (0,1)\times (0,1)),$$
where the first component of the disjoint union belongs to $\mathscr{N}_{sing}(X)$ and the second one is the open subset $(\RR^{N-1}\times (0,1))\times (0,1)$ of  $\RR^N\times (0,1)$. For introduce $U_\mathscr{F}^2$, we will need to chose a (good) tubular neighborhood of $X$ in $\RR^N$, that is a diffeomorphism 
\begin{equation}\label{tubf}
\xymatrix{
N(X) \ar[r]^-{f}_-\approx& W\subset\RR^{N-1}\times \RR_-\\
X\ar[u]\ar[r]^-{id}&X\ar[u]
}
\end{equation}
from $N(X)$ to $W$ an open neighborhood of $X$ (or $i_\partial (X)$ to be formal) in $\RR^N$ such that $f$ is the identity in $X$ (identifying $X$ with the zero section), and the restriction to $\partial X$ gives a tubular neighborhood  of $\partial X$ in $\RR^{N-1}$:
\begin{equation}\label{tubfb}
\xymatrix{
N(X)|_{\partial X} \ar[r]^-{f_\partial}_-\approx& W_\partial \subset\RR^{N-1}\times \{0\}\\
\partial X\ar[u]\ar[r]^-{id}&\partial X\ar[u]
}
\end{equation}
where $W_\partial= W\bigcap \RR^{N-1}\times \{0\}$. Consider the open subset of $\RR^N$ consisting of putting a collar to $W$ in the positive direction:

$$W_{sing}:=W\bigcup W_\partial \times [0,1)$$

We will introduce the subset $U_\mathscr{F}^2$ by its intersections with $\mathscr{N}_{sing}(X)$ and with $\RR^N\times (0,1]$:

$$U_\mathscr{F}^2\bigcap \mathscr{N}_{sing}(X)= N(X)\times \{0\}\bigsqcup W_\partial \times (0,1),$$
and
$$U_\mathscr{F}^2\bigcap \RR^N\times (0,1]=W_{sing}\times (0,1).$$



\begin{lemma}\label{openBF}
The subsets $U_\mathscr{F}^0$, $U_\mathscr{F}^1$ and $U_\mathscr{F}^2$ of $\mathscr{B}_\mathscr{F}$ are open.
\end{lemma}

\begin{proof}
By definition $\mathscr{B}_\mathscr{F}$ induced from the quotient topology on $\mathscr{B}_h$. Hence if we consider the map
$$q_\mathscr{F}:X_\mathscr{F}\times \RR^N\to \mathscr{B}_\mathscr{F}$$
given by $q_\mathscr{F}:=q_\partial\bigsqcup \overset{\:\circ}{q}$ (with the notations of lemmas \ref{lemmaqint} and \ref{lemmaqbord}), it is enough to observe that $q_\mathscr{F}^{-1}(V)$ is an open subset of $X_\mathscr{F}\times \RR^N$ to conclude that $V$ is open in $\mathscr{B}_\mathscr{F}$. In the case we are dealing with, we have:
\begin{itemize}
\item[$\bullet$] $q_\mathscr{F}^{-1}(U_\mathscr{F}^0)=\intx \times (0,1]\times \RR^N,$ which is clearly open in 
$X_\mathscr{F}\times \RR^N$.
\item[$\bullet$] $q_\mathscr{F}^{-1}(U_\mathscr{F}^1)=X \times (0,1)\times \RR^N,$ which is clearly open in 
$X_\mathscr{F}\times \RR^N$.
\item[$\bullet$] For the last one, an explicit description of the inverse image allows to verify it is open by a direct analysis at each point, we leave this as direct excercice computation,
$$q_\mathscr{F}^{-1}(U_\mathscr{F}^2)=
X\times \{0\}\times \RR^N$$
$$\bigsqcup\{(x,\varepsilon,X)\in \partial X\times (0,1)\times \RR^N: i_\partial(x)+\varepsilon\cdot X_{N-1}\in W_\partial \}$$
$$\bigsqcup \{(x,\varepsilon,X)\in X\times (0,1)\times\RR^N: i(x)+\varepsilon\cdot X\in W_{sing} \}$$
\end{itemize}
\end{proof}

\begin{proposition}\label{BForiented}
The locally compact space $\mathscr{B}_\mathscr{F}$ admits an oriented  $C^\infty$-manifold with boundary structure of dimension $N+1$. 
\end{proposition}
\begin{proof}
We will cover $\mathscr{B}_\mathscr{F}$ with three explicit charts by using the three open subsets of $\mathscr{B}_\mathscr{F}$, $U_\mathscr{F}^0$, $U_\mathscr{F}^1$ and $U_\mathscr{F}^2$ of the lemma above. 

\vspace{2mm}

{\bf Chart $(U_\mathscr{F}^0, \Psi_{\mathscr{F}}^{0})$:} 
We let $W_\mathscr{F}^0=\RR^N\times (0,1]$ and $\Psi_{\mathscr{F}}^{0}$ to be the identity:
$$W_\mathscr{F}^0\stackrel{\Psi_{\mathscr{F}}^{0}=id}{\longrightarrow} U_\mathscr{F}^0.$$

{\bf Chart $(U_\mathscr{F}^1, \Psi_{\mathscr{F}}^{1})$:} 
We let $W_\mathscr{F}^1=\RR^{N-1}\times [0,1)_t\times (0,1)_s$ and $\Psi_{\mathscr{F}}^{1}$, 
$$W_\mathscr{F}^1\stackrel{\Psi_{\mathscr{F}}^{1}}{\longrightarrow} U_\mathscr{F}^1$$
defined by
\begin{equation}
\begin{cases}
(X_{N-1},0,s)\mapsto (X_{N-1},s)\in \RR^{N-1}\times (0,1)\subset \mathscr{N}_{sing}\\
(X_{N-1},t,s)\mapsto (X_{N-1},s,t)\in (\RR^{N-1}\times (0,1))\times (0,1)\subset \RR^N\times (0,1].
\end{cases}
\end{equation}

{\bf Chart $(U_\mathscr{F}^2, \Psi_{\mathscr{F}}^{2})$:} 
We let $W_\mathscr{F}^2=W_{sing}\times [0,1)_t$ and $\Psi_{\mathscr{F}}^{2}$, 
$$W_\mathscr{F}^2\stackrel{\Psi_{\mathscr{F}}^{2}}{\longrightarrow} U_\mathscr{F}^2$$
defined by
\begin{equation}
\begin{cases}
(w,0)\in W\mapsto (f^{-1}(w),0)\in N(X)\times \{0\}\subset \mathscr{N}_{sing}\\
(w_\partial,s,0)\mapsto (f_\partial(s\cdot f_{\partial}^{-1}(w_\partial)),s)\in W_\partial \times (0,1) \subset \mathscr{N}_{sing}.
\end{cases}
\end{equation}
if $t=0$, and by
\begin{equation}
\begin{cases}
(w,t)\mapsto (f(t\cdot f^{-1}(w)),t)\in W\times (0,1)\subset \RR^N\times (0,1] \\
(w_\partial,s,t)\mapsto (f_\partial((s+t)\cdot f_{\partial}^{-1}(w_\partial)),s,t)\in (W_\partial\times [0,1))\times (0,1)\subset \RR^N\times (0,1].
\end{cases}
\end{equation}
for $t\neq 0$.

We will check now the compatibility of the charts, together with the fact that the changes of coordinates have positive sign:
\begin{itemize}
\item[$\bullet$] $((\Psi_{\mathscr{F}}^{0})^{-1}\circ \Psi_{\mathscr{F}}^{1})$: 
This is the easiest case, indeed we have that
$$(\Psi_{\mathscr{F}}^{1})^{-1}(U_\mathscr{F}^0\bigcap U_\mathscr{F}^1)=(\RR^{N-1}\times (0,1))\times (0,1),$$ $$(\Psi_{\mathscr{F}}^{0})^{-1}(U_\mathscr{F}^0\bigcap U_\mathscr{F}^1)=(\RR^{N-1}\times (0,1))\times (0,1)$$ and
$$
\xymatrix{
(\Psi_{\mathscr{F}}^{1})^{-1}(U_\mathscr{F}^0\bigcap U_\mathscr{F}^1)\ar[rr]^-{(\Psi_{\mathscr{F}}^{0})^{-1}\circ \Psi_{\mathscr{F}}^{1}}&&(\Psi_{\mathscr{F}}^{0})^{-1}(U_\mathscr{F}^0\bigcap U_\mathscr{F}^1)
}
$$
is the identity.

\item[$\bullet$] $((\Psi_{\mathscr{F}}^{0})^{-1}\circ \Psi_{\mathscr{F}}^{2})$: 

$$(\Psi_{\mathscr{F}}^{2})^{-1}(U_\mathscr{F}^0\bigcap U_\mathscr{F}^2)=W_{sing}\times (0,1),$$ $$(\Psi_{\mathscr{F}}^{0})^{-1}(U_\mathscr{F}^0\bigcap U_\mathscr{F}^2)=W_{sing}\times (0,1)$$ and
$$
\xymatrix{
(\Psi_{\mathscr{F}}^{2})^{-1}(U_\mathscr{F}^0\bigcap U_\mathscr{F}^2)\ar[rr]^-{(\Psi_{\mathscr{F}}^{0})^{-1}\circ \Psi_{\mathscr{F}}^{2}}&&(\Psi_{\mathscr{F}}^{0})^{-1}(U_\mathscr{F}^0\bigcap U_\mathscr{F}^2)
}
$$
is given by
$$
\begin{cases}
(w,t)\mapsto (f(t\cdot f^{-1}(w)),t)
\\
(w_\partial,s,t)\mapsto (f_\partial((s+t)\cdot f_\partial^{-1}(w_\partial)),s,t)
\end{cases}
$$
which is evidently a diffeomorphism with positive determinant\footnote{It is a diffeomorphism whose linear representation is of the form $A\cdot\lambda \cdot A^{-1}$ with positive $\lambda$.}.

\item[$\bullet$] $((\Psi_{\mathscr{F}}^{1})^{-1}\circ \Psi_{\mathscr{F}}^{2})$: 

$$(\Psi_{\mathscr{F}}^{2})^{-1}(U_\mathscr{F}^1\bigcap U_\mathscr{F}^2)=W_\partial \times (0,1)\times [0,1),$$ 
$$(\Psi_{\mathscr{F}}^{1})^{-1}(U_\mathscr{F}^1\bigcap U_\mathscr{F}^2)=W_\partial \times [0,1)\times (0,1)$$ and
$$
\xymatrix{
(\Psi_{\mathscr{F}}^{2})^{-1}(U_\mathscr{F}^1\bigcap U_\mathscr{F}^2)\ar[rr]^-{(\Psi_{\mathscr{F}}^{1})^{-1}\circ \Psi_{\mathscr{F}}^{2}}&&(\Psi_{\mathscr{F}}^{1})^{-1}(U_\mathscr{F}^1\bigcap U_\mathscr{F}^2)
}
$$
is given by
$$(w_\partial,s,t)\mapsto (f_\partial((s+t)\cdot f_\partial^{-1}(w_\partial)),t,s)$$
which is evidently a diffeomorphism with positive determinant\footnote{same as previous footnote.}.
\end{itemize}

\end{proof}

In particular, we have obtained an oriented cobordism $\mathscr{B}_\mathscr{F}$ from $\mathscr{N}_{sing}(X)$ to $\RR^N$. From now on, we orient $\mathscr{B}_\mathscr{F}$ such that the induced orientation on the boundary is $$\partial \mathscr{B}_\mathscr{F}=-\Nb_{sing}(X)\bigcup \RR^N.$$
We can hence apply a Stoke's theorem argument to obtain the following proposition:

\begin{proposition}\label{Stokes}
The following diagram is commutative
$$
\xymatrix{
K^0(\mathscr{N}_{sing}(X))\ar@/^2pc/[rr]^-{ind_{\mathscr{B}_\mathscr{F}}}\ar[rd]_{\int_{\mathscr{N}_{sing}(X)}ch(\cdot)}&
K^0(\mathscr{B}_\mathscr{F})\ar[l]_-{(e_0)_*}^-{\approx}\ar[r]^-{(e_1)_*}& K^0(\mathbb{R}^N)\ar[dl]^{\int_{\RR^N}ch(\cdot)}\\
&\RR&
}
$$
\end{proposition}

\begin{proof}
By definition, the algebra morphisms $e_0:C_0(\mathscr{B}_{\mathscr{F}})\to C_0(\Nb_{sing}(X))$  and $e_1:C_0(\mathscr{B}_{\mathscr{F}})\to C_0(\RR^N)$ are induced by the closed embeddings $i_0:\Nb_{sing}(X)\hookrightarrow \mathscr{B}_{\mathscr{F}}$ and $i_1:\RR^N\hookrightarrow \mathscr{B}_{\mathscr{F}}$ respectively.
The Chern character being natural we have that the following diagram is commutative 
\begin{equation} \xymatrix{ K^0(\mathscr{N}_{sing}(X))\ar[d]_{ch}& K^0(\mathscr{B}_\mathscr{F})\ar[d]_{ch}\ar[l]_-{(e_0)_*}^-{\approx}\ar[r]^-{(e_1)_*}& K^0(\mathbb{R}^N)\ar[d]_{ch}\\ H^{ev}_{dR}(\mathscr{N}_{sing}(X))& H^{ev}_{dR}(\mathscr{B}_\mathscr{F})\ar[l]^-{(i_0)_*}\ar[r]_-{(i_1)_*}& H^{ev}_{dR}(\mathbb{R}^N) } \end{equation}
The result now follows from Stoke's theorem, indeed, for any $\omega$ $N$-closed differential form on $\mathscr{B}_\mathscr{F}$ with compact support, we have by Stoke's that
$$\int_{\partial \mathscr{B}_\mathscr{F}}\omega=O$$
and hence $\int_{\Nb_{sing}(X)}\omega=\int_{\RR^N}\omega$.
\end{proof}

\section{Atiyah-Patodi-Singer theorem in K-theory}
\subsection{The Fredholm index morphism}\label{fim}

Deformation groupoids induce index morphisms. The groupoid $\Gamma(X)_{\cF}$ is parametrized by the closed interval $[0,1]$. Its algebra comes 
equipped with evaluations to the algebra of $\mathcal{T}_{nc}X$ (at t=0) and to the algebra of $\intx\times \intx$ (for $t\neq 0$). We have a short exact sequence of $C^*$-algebras
\begin{equation}\label{btangentsuite}
0\to C^*(\intx\times \intx\times (0,1])\to C^*(\Gamma(X)_{\cF})\to
C^*(\mathcal{T}_{nc}X)\to 0
\end{equation}
where the algebra $C^*(\intx\times \intx\times (0,1])$ is contractible. Hence applying the $K$-theory functor to this sequence we obtain an index morphism
$$ind_{\cF}=(e_1)_*\circ(e_0)_{*}^{-1}:K_0(C^*(\mathcal{T}_{nc}X))\longrightarrow
K_0(C^*(\intx\times \intx))\approx \mathbb{Z}.$$

\begin{proposition}\label{FAPS}\cite{LMNpdo,Les09,DLR}.
For any fully elliptic operator $D$ on $X$, there is a "noncommutative symbol" $[\sigma_D]\in K_0(C^*(\mathcal{T}_{nc}X))$ and 
\begin{equation}
ind_{\cF}(\sigma_D)= \mathrm{Index}_{APS}(D)
\end{equation}
\end{proposition}
\begin{proof} 
For the sake of completeness, we briefly explain the proof. Let $P\in\Psi^0(\Gamma(X);E,F)$ be a zero-order fully elliptic $b$-operator. Here $E,F$ are hermitian bundles on $X$,
pulled-back to $\Gamma(X)$ with the target map. Let $Q\in\Psi^0(\Gamma(X);F,E)$ be a full parametrix of $P$. This means:
\begin{equation}\label{int-par-cond}
 PQ-1\in\Psi^{-1}(\Gamma(X);F,F),\quad QP-1\in\Psi^{-1}(\Gamma(X);E,E)
\end{equation}
and that, moreover, 
\begin{equation}\label{bound-par-cond}
P_\partial Q_\partial = 1,\quad  Q_\partial P_\partial  = 1  
\end{equation}
where we have denoted by $A_\partial$ the $\partial X\times\partial X\times
\RR$-pseudodifferential operator obtained by restriction of any $\Gamma(X)$-operator $A$. 
The equations (\ref{int-par-cond}) reflect the interior ellipticity while (\ref{bound-par-cond}) reflect the boundary ellipticity. 
It is well known that $ P : L^2_b(X,E)\longrightarrow L^2_b(X,F)$ is bounded and Fredholm \cite{Mel,Mont} where the hermitian structure of $E,F$ as bundles
over $X$ is used together with the measure on $X$ associated with a $b$-metric, and 
 $$\mathrm{Index}_{APS}(P) = \mathrm{Index}_{\mathrm{Fred}}(P)= \mathrm{dim} \ker P-\mathrm{dim}\mathrm{coker}P.$$
This can also be recovered as the analytical index of a $K$-homology class of a compact space. 
Let $X^c=X/\partial X$ be the conical space
associated with $X$ and we note $\pi$ the canonical projection map. We represent $C(X^c)$ into $L^2_b(X,G)$, where $G$ is any hermitian bundle over $X$, as
follows:
 $$ f\in C(X^c),\quad \xi\in L^2_b(X,G); \quad \forall y\in X,\quad m(f)(\xi)(y)=f(\pi(y))\xi(y).$$
It is immediate to check that
\begin{equation}
 (P)= \left( L^2_b(X,E\oplus F); m ; \begin{pmatrix} 0 & Q\\ P& 0 \end{pmatrix} \right)
\end{equation}
is a Kasparov $(C(X^c),\CC)$-module and that the resulting $K$-homology class $[P]\in K_0(X^c)$ does not depend on the choices of the parametrix $Q$ nor the
particular $b$-metric. Then $\mathrm{Index}_{\mathrm{APS}}(P)=\mathrm{Index}_{\mathrm{ana}}(P)=p_*([P])$ where $p : X^c\to\{\mathrm{point}\}$. 

We define $\sigma_P\in K_0(C^*(\mathcal{T}_{nc}X))$ to be the Poincar\'e dual class of $P$. Let us describe this element more explicitly. 

 Let  $\widetilde{P}$,  $\widetilde{Q}$ be any  elliptic operators on $\Gamma(X)^{tan}$ such that $\widetilde{P}|_{t=1} = P$, $\widetilde{Q}|_{t=1} = Q$ 
and:
 $$\widetilde{P} \widetilde{Q}-1, \widetilde{Q}\widetilde{P}-1\in\Psi^{-1}(\Gamma(X)^{tan}). $$
We have by construction:
$$
(\widetilde{P}_{t=1})_{\partial}(\widetilde{Q}_{t=1})_{\partial}-1=0=(\widetilde{Q}_{t=1})_{\partial}(\widetilde{P}_{t=1})_{\partial}-1,$$
Hence:
\begin{equation}
 (\widetilde{P})=\left( C^*(\Gamma(X)_{\mathcal{F}},E\oplus F),1, \begin{pmatrix} 0 & \widetilde{Q}\\ \widetilde{P} & 0 \end{pmatrix} \right)
\end{equation}
is a Kasparov $(\CC,C^*(\Gamma(X)_{\mathcal{F}}))$-module, and the restriction: 
\begin{equation}
 \sigma_{nc}(\widetilde{P}): =\widetilde{P}|_{\mathcal{T}_{nc}X}
\end{equation}
 provides a Kasparov $(\CC,C^*(\mathcal{T}_{nc}X))$-module:
 \begin{equation}
 (\sigma_{nc}(\widetilde{P})): =\left(C^*(\mathcal{T}_{nc}X,E\oplus F),1, \begin{pmatrix} 0 & \sigma_{nc}(\widetilde{Q}) \\
\sigma_{nc}(\widetilde{P})  & 0 \end{pmatrix} \right).
\end{equation}
Denoting by $e_0$  the $*$-homomorphism $C^*(\Gamma(X)_{\mathcal{F}})\to C^*(\mathcal{T}_{nc}X)$, the previous class satisfies:
\begin{equation}
 (e_0)_* [\widetilde{P}]=[\sigma_{nc}(\widetilde{P})] \in K_0(C^*(\mathcal{T}_{nc}X))
\end{equation}
By construction the Poincar\'e duality isomorphism \cite{DL09,Les09} sends $[P]$ to $[\sigma_{nc}(\widetilde{P})]$, and we thus set:
 $$ \sigma_P = [\sigma_{nc}(\widetilde{P})]$$
Now, $\mu_0$ denoting the Morita equivalence $\overset{\circ}{X}\times\overset{\circ}{X}\sim  \mathrm{point}$, we compute: 
\begin{eqnarray}
 ind_{\cF}(\sigma_P)& = & \mu_0\circ (e_1)_*\circ (e_0)_*^{-1}([\sigma_{nc}(\widetilde{P})] \\
  &=& \mu_0\circ(e_1)_*(\widetilde{P})\\
  &=& \mu_0\left( C^*(\overset{\circ}{X}\times\overset{\circ}{X},E\oplus F),1, \begin{pmatrix} 0 & Q\\ P& 0 \end{pmatrix}\right)\\
  &=& \left( L^2_b(X,E\oplus F),1, \begin{pmatrix} 0 & Q\\ P& 0 \end{pmatrix}\right)\\
   &=& \mathrm{Index}_{\mathrm{ana}}(P)=p_*([P])=\mathrm{Index}_{\mathrm{APS}}(P)
\end{eqnarray}
\end{proof}
It is also interesting to manage geometric operators (for instance, Dirac type operators  on $X$ equipped with an exact $b$-metric $g_b$) instead of abstract $0$-order pseudodifferential operators. Under appropriate assumptions, they also give rise to $K$-homology classes of $X^c$ and thus one may look for a {\sl geometric} representative of their Poincar\'e dual class in $K_0( C^*(\mathcal{T}_{nc}X))$. 

Rather than dealing exaclty with Dirac operators on $(X,g_b)$, we shall consider the following class of differential operators on $X$ containing them. 
Let $E$ be a smooth hermitian vector bundle over $X$ endowed with an orthogonal decomposition $E=E_0\oplus E_1$ and an isomorphism $ U : E|_{\partial X\times (0,\varepsilon)}\to (p_1)^*E_\partial \oplus (p_1)^*E_\partial $ where $p_1$ is the first projection subordonated to a collar identification near the boundary and $E_\partial=E_0|_{\partial X}$. 
Then we consider first order elliptic differential operators $D$, which have, after conjugation by $U$, the following expression near the boundary:
\begin{equation}\label{model-dirac-boundary}
U  D U^{-1}= \begin{pmatrix} 0 & -x\frac{\partial}{\partial x} + S \\ x\frac{\partial}{\partial x} + S & 0 \end{pmatrix}= \begin{pmatrix} 0 & D_- \\ D_+ & 0 \end{pmatrix}
\end{equation}
where $S\in\mathrm{Diff}^1(\partial X,E_\partial)$. We require $S$ to be elliptic, symmetric, and  to simplify independant of $x<\varepsilon$.
It follows from these assumptions that $D^2+1$ is invertible, as a linear map, on $C^\infty(X,E)$ \cite{Mel} and since $\Psi(\Gamma(X))$ is spectrally invariant (it is understood that the appropriate Schwartz algebra is added in the calculus \cite{lmn}), we have $(D^2+1)^{-1/2}\in \Psi^{-1}(\Gamma(X))$. Now, we moreover require the invertibility of $S$, which is equivalent here to the full ellipticity of $D$. Observe that this assumption is not sufficient to let $D$ into an unbounded $(C(X^c),\CC)$-Kasparov module in the sense of \cite{bj}, but it nevertheless implies that $W(D):=D(D^2+1)^{-1/2}\in \Psi^{0}(\Gamma(X),E)$ is fully elliptic since the indicial family map $I : \Psi(\Gamma(X))\to \Psi(\partial X,\RR)$ is a homomorphism of algebras:
 $$ I(W(D),\tau) = \begin{pmatrix} 0 & \frac{S-i\tau}{(S^2+\tau^2+1)^{1/2}} \\  \frac{S+i\tau}{(S^2+\tau^2+1)^{1/2}} & 0 \end{pmatrix}= \begin{pmatrix} 0 &  W(D)_- \\  W(D)_+ & 0 \end{pmatrix}. $$
Thus one can already associate to $D$ the (bounded) $K$-homology class $[W(D)_+]$, with Poincar\'e dual given by the (bounded) $K$-theory class $\sigma_{W(D)_+}$ as above. Alternatively, we may look for a more geometric representative of these classes.


For that purpose, we define a lift
$\widetilde{D}\in\Psi(\Gamma(X)^{tan})$ of $D$ as follows. Let $s$ be
the complete symbol of $S$ with respect to the exponential map of the
metric on the boundary (see \cite{getzler}), and $d$ be the complete symbol of $D$ with respect to the exponential map of the metric on $X$.  We rescale $s$ and $d$ as follows: 
$$ \forall 0<\varepsilon\le 1, \forall (x,\xi)\in T^*\partial X,\quad s_{\ad}(y,\xi)=s(y,\varepsilon \xi) \text{ and } \forall (z,\zeta)\in T^*X,\quad d_{\ad}(z,\zeta)=d(z,\varepsilon \zeta)$$      
Setting $S_{\ad}=s_{\ad}(y,D_y)$, $D_{\ad}=d_{\ad}(z,D_z)$ and using positive functions $\varphi,\psi\in C^\infty(X)$ such that $\varphi+\psi=1$ and $\varphi=1$ if $x<\varepsilon$, 
 $\varphi=0$ if $x\ge 2\varepsilon$, we let 
$$\widetilde{D}|_{x\varepsilon>0}= \varphi U^{-1}\begin{pmatrix} 0 & -\varepsilon x\frac{\partial}{\partial x} + S_{\ad} \\ \varepsilon x\frac{\partial}{\partial x} + S_{\ad} & 0 \end{pmatrix}U + \psi D_{\ad},$$ 
$$\widetilde{D}|_{x=0,\varepsilon>0}= U^{-1}\begin{pmatrix} 0 & -\frac{\partial}{\partial \lambda} + S_{\ad} \\ \frac{\partial}{\partial \lambda} + S_{\ad} & 0 \end{pmatrix}U,$$ 
and 
$$\widetilde{D}|_{\varepsilon=0} = \varphi U^{-1}\begin{pmatrix} 0 & -\frac{\partial}{\partial \lambda} + s(y,D_Y) \\ \frac{\partial}{\partial \lambda} + s(y,D_Y) & 0 \end{pmatrix}U + \psi d(z,D_Z),$$ 
where $D_Y,D_Z$ stand for the differentiation in the fibers coordinates of $T\partial X$ and $TX$ respectively. 

On the other hand, we can also rescale $\widetilde{D}|_{x=0,\varepsilon>0}$ into:
$$ \sigma_{\partial-\ubd}(D)= U^{-1}\begin{pmatrix} 0 & -\frac{\partial}{\partial \lambda} + \frac1{1-\varepsilon} S_{\ad} \\ \frac{\partial}{\partial \lambda} + \frac1{1-\varepsilon}S_{\ad} & 0 \end{pmatrix}U$$
and define  $\sigma_{\ubd}(D)\in \mathrm{Diff}^1(\mathcal{T}_{nc}X)$ by 
$$\sigma_{\ubd}(D) = \sigma_{\partial-\ubd}(D) \text{ if } x=0 \text{ and } \sigma_{\ubd}(D) =\widetilde{D}|_{\varepsilon=0}\text{ if } x>0 $$ 
We also note $\widetilde{E}$  the pull back of $E$ for the map  $X\times [0,1]\overset{p_1}{\longrightarrow} X$ and then $E_{\ad}$ its restriction to $\{ x\varepsilon=0\}$. By construction, 
we get an elliptic symmetric element $\sigma_{\ubd}(D)\in \mathrm{Diff}^1(\mathcal{T}_{nc}X,r^*E_{\ad})$, whose closure $\overline{\sigma_{\ubd}(D)}$ as an unbounded operator on the Hilbert module $\cE:=C^*(\mathcal{T}_{nc}X,r^*E_{\ad})$ with domain $C^{\infty,0}(\mathcal{T}_{nc}X,r^*E_{\ad})$ is a regular selfadjoint operator \cite{bj}. To prove that assertion, we can not directly apply proposition 3.6.2 and lemma 3.6.3 in \cite{VassoutJFA} since the unit space of $\mathcal{T}_{nc}X$ is not compact. Nevertheless, we can pick up a suitable parametrix $q\in \Psi^{-1}(\overline{\mathcal{T}_{nc}X},r^*E_{\ad})$ of  $\sigma_{\ubd}(D)$, where  $\overline{\mathcal{T}_{nc}X}:=\Gamma(X)^{tan}|_{x\varepsilon=0}$, in such a way that the proofs given in \cite{VassoutJFA} apply verbatim : we define $q$ by combining, using cut-off functions, a parametrix given by the inversion of the principal symbol of $\sigma_{\ubd}(D)$  with the true inverse of $ \sigma_{\partial-\ubd}(D)|_\varepsilon$ when $1-\alpha<\varepsilon<1$ and extended by $0$ at $\varepsilon=1$. Then, we have by construction:
$$\sigma_{\ubd}(D)q=1+k_1,\ q\sigma_{\ubd}(D)=1+k_2  \text{ with } q,k_i\in \Psi^{-1}(\overline{\mathcal{T}_{nc}X},r^*E_{\ad}),\ q|_{\varepsilon=1}=0,\ k_i|_{\varepsilon=1}=0.$$
The first of the two conditions on  the operators $q,k_1,k_2$ implies that they extend into compact morphisms on $C^*(\overline{\mathcal{T}_{nc}X},r^*E_{\ad})$ and the second that they actually are compact on $C^*(\mathcal{T}_{nc}X,r^*E_{\ad})$. 

It then follows that $(\sigma_{ubd}(D)^2+1)^{-1}$ is a compact morphism on $C^*(\mathcal{T}_{nc}X,r^*E_{\ad})$. Thus, we get an unbounded $(\CC,C^*(\mathcal{T}_{nc}X))$-Kasparov class in the sense of \cite{bj}:
$$ \sigma_{\ubd,D}:=\left(C^*(\mathcal{T}_{nc}X,r^*E_{\ad}\right),\sigma_{ubd}(D))\in \cE_{\ubd}(\CC, C^*(\mathcal{T}_{nc}X)).$$
In the equality above, $\cE_{\ubd}(A,B)$ denotes the family of unbounded Kasparov $A$-$B$-bimodules as defined in \cite{bj}.

 To check that the latter is an unbounded representative of  the Poincar\'e dual of $[W(D)_+]$ (and thus can be used for the computation of the index of $D_+$), we have to prove the equality :
$$  [\sigma_{W(D_+)}] = [W(\sigma_{ubd}(D))]\in KK(\CC, C^*(\mathcal{T}_{nc}X)),$$ 
which can be achieved by comparing the operator part of these $KK$-classes. When $\varepsilon$, they coincide. 
When $\varepsilon>0$, the operator part in $[\sigma_{W(D_+)}]$ can be represented by:
$$ \begin{pmatrix} 0 & \frac{S_{\ad}-i\partial_\lambda}{(S_{\ad}^2+\partial_\lambda^2+1)^{1/2}} \\  \frac{S_{\ad}+i\partial_\lambda}{(S_{\ad}^2+\partial_\lambda^2+1)^{1/2}} & 0 \end{pmatrix},$$
and for $[W(\sigma_{ubd}(D))]$  we have:   
$$  \begin{pmatrix} 0 & \frac{S_{\ad}-i(1-\varepsilon)\partial_\lambda}{(S_{\ad}^2+(1-\varepsilon)^2(\partial_\lambda^2+1))^{1/2}} \\  \frac{S_{\ad}+i(1-\varepsilon)\partial_\lambda}{(S_{\ad}^2+(1-\varepsilon)^2(\partial_\lambda^2+1))^{1/2}} & 0 \end{pmatrix}.$$
Homotoping the numerical factor $(1-\varepsilon)$ with $1$ provides an operator homotopy between both, and this proves the assertion.  

Observe also that: 
$$W(\sigma_{ubd}(D))|_{\varepsilon=1} = \begin{pmatrix} 0 & \frac{S}{|S|} \\  \frac{S}{|S|} & 0 \end{pmatrix},$$
and thus that, playing again with homotopies, this value at $\varepsilon=1$ can be conserved for $\varepsilon\in]\alpha,1]$ for arbitrary $\alpha>0$.  

\subsection{Atiyah-Patodi-Singer index theorem in K-theory}

\begin{definition}\label{deftop}[Atiyah-Patodi-Singer topological index morphism for a manifold with boundary]
Let $X$ be a manifold with boundary consider an embedding of $X$ in
$\RR^N$ as in  Section \ref{sectionembedding}. The topological index morphism of $X$ is the morphism
$$ind_t^X:K_0(C^*(\mathcal{T}_{nc}X))\longrightarrow \mathbb{Z}$$
defined  as the composition of the following three morphisms
\begin{enumerate}
\item The Connes-Thom isomorphism $\mathscr{CT}_0$ followed by the Morita equivalence $\mathscr{M}_0$:
$$K_0(C^*(\mathcal{T}_{nc}X))\stackrel{\mathscr{CT}_0}{\longrightarrow}K_0(C^*((\mathcal{T}_{nc}X)_{h_0}))\stackrel{\mathscr{M}_0}{\longrightarrow}K^0(\mathscr{N}_{sing}(X))$$
\item The index morphism of the deformation space $\mathscr{B}_\mathscr{F}$ (proposition \ref{Stokes}):
$$
\xymatrix{
K^0(\mathscr{N}_{sing}(X))\ar@/^2pc/[rr]^-{ind_{\mathscr{B}_\mathscr{F}}}&
K^0(\mathscr{B}_\mathscr{F})\ar[l]_-{(e_0)_*}^-{\approx}\ar[r]^-{(e_1)_*}& K^0(\mathbb{R}^N)
}
$$
and
\item the usual Bott periodicity morphism:
$K^0(\mathbb{R}^N)\stackrel{Bott}{\longrightarrow}\mathbb{Z}.$\\
\end{enumerate}

In other terms, the topological index fits by definition in the following commutative diagram

$$\xymatrix{
 K_0(C^*(\mathcal{T}_{nc}X)) \ar[d]_-{\mathscr{CT}}^-{\approx} \ar[rrr]^{ind_t^X}&&& \mathbb{Z} \\
K^0(\mathscr{N}_{sing}(X)) \ar[rrr]_-{ind_{\mathscr{B}_\mathscr{F}}} &&& K^0(\mathbb{R}^N)\ar[u]_-{Bott}^-{\approx}\\
}$$

\end{definition}

\begin{remark}
The topological index defined above is a natural generalisation of the topological index theorem defined by Atiyah-Singer. Indeed, in the smooth case, they coincide. 

\end{remark}

We now prove, as it was outlined in \cite{CMont}, the index theorem.

\begin{theorem}\label{KAPS}[K-theoretic APS]
Let $X$ be a manifold with boundary, consider an embedding of $X$ in
$\RR^N$ as in  Section \ref{sectionembedding}. The Fredholm index
equals the topological index.
\end{theorem}

\begin{proof}
The morphism $h:\Gamma(X)_{\cF}\rightarrow \mathbb{R}^N$ is by definition also parametrized by $[0,1]$, {\it i.e.}, we have morphisms $h_0:\mathcal{T}_{nc}X\rightarrow \mathbb{R}^N$ and 
$h_1:\intx\times \intx \rightarrow \mathbb{R}^N$, for $t=1$. We
can consider  the associated groupoids, which  are free and proper.

The following diagram, in which the morphisms $\mathscr{CT}$ and $\mathscr{M}$ are the Connes-Thom and Morita isomorphisms respectively, is trivially commutative by naturality of the Connes-Thom isomorphism:

{\small \begin{equation}\label{diagAPS}
 \xymatrix{
K_0(C^*(\mathcal{T}_{nc}X)) \ar[d]^-{\approx}_{\mathscr{CT}}& K_0(C^*(\Gamma(X)_{\cF})) \ar[d]^-{\approx}_{\mathscr{CT}} \ar[l]_-{e_0}^-{\approx}\ar[r]^-{e_1}& 
K_0(C^*(\intx\times \intx))\ar[d]^-{\approx}_{\mathscr{CT}}\\
K_0(C^*((\mathcal{T}_{nc}X)_{h_0})) \ar[d]^-{\approx}_{\mathscr{M}}& K_0(C^*((\Gamma(X)_{\cF})_h)) \ar[d]^-{\approx}_{\mathscr{M}}\ar[l]_-{e_0}^-{\approx}\ar[r]^-{e_1}& K_0(C^*((\intx\times \intx)_{h_1}))\ar[d]^-{\approx}_{\mathscr{M}}\\
K^0(\mathscr{N}_{sing}(X)) & K^0(\mathscr{B}_\mathscr{F}) \ar[l]_-{e_0}^-{\approx}\ar[r]^-{e_1} & 
K^0(\mathbb{R}^N).
}
\end{equation}}

The left vertical line gives the first part of the topological index
 map. The bottom line is the morphism induced by the
 deformation space $\mathscr{B}$. And the right vertical line is precisely the inverse of the Bott isomorphism $\mathbb{Z}=K^0(\{pt\})\approx K_0(C^*(\intx\times \intx))\rightarrow K^0(\mathbb{R}^N)$. Since the top
 line gives $ind_{\cF}$, we obtain the result.
\end{proof}

\begin{corollary}
  The topological index does not depend on the choice of the embedding.
\end{corollary}

\section{The cohomological APS formula}

The theorem \ref{KAPS} (see also diagram (\ref{diagAPS})) tells us that the computation of the index can be performed (modulo Connes-Thom and Morita) as the computation of the index of a deformation space :
$$
\xymatrix{
K^0(\mathscr{N}_{sing}(X))&
K^0(\mathscr{B}_\mathscr{F})\ar[l]_-{(e_0)_*}^-{\approx}\ar[r]^-{(e_1)_*}& K^0(\mathbb{R}^N)
}
$$

Now, consider the following diagram
\begin{equation}
\xymatrix{
K^0(\RR^N)\ar[r]^-{Ch}\ar[d]_-{Bott}&H^*_{dR}(\RR^N)\ar[d]^-{\int_{\RR^N}\cdot}\\
\ZZ\ar[r]&\CC,
}
\end{equation}
where $\int_{\RR^N}$ is the integration with respect to the fundamental class of $\RR^N$. It is well known that this diagram is commutative.

We can summarize the previous statements in the following result, 
which is then an immediate consequence of theorem \ref{KAPS} (see again diagram (\ref{diagAPS})):

\begin{corollary}\label{APScohomologyformula}
Let $(X,\partial X)$ be a manifold with boundary, and let
$i:X\hookrightarrow \RR^N$ be an embedding as in Section
\ref{sectionembedding}; we use the notations of last Sections.

The index morphism $ind_\cF$ fits in the following commutative diagram
\begin{equation}
\xymatrix{
&K_0(C^*(\mathcal{T}_{nc}(X)))\ar[rrdd]_-{ind_\cF}\ar[r]^-{\mathscr{CT}_{h_0}}&K_0(C^*(\mathcal{T}_{nc}(X))_{h_0}))\ar[r]^ -{Morita}&
K^0(\mathscr{N}_{sing}(X))\ar[d]^{Ch}&\\ &&&H^*(\mathscr{N}_{sing}(X))\ar[d]^-{\int_{\mathscr{N}_{sing}(X)}}&\\
&&&\CC&
}
\end{equation}
For keeping short notations we will denote by $\mathscr{CT}$ the composition of $\mathscr{CT}_h$ with the Morita equivalence induced isomorphism $\mathscr{M}$.

In particular, for any fully elliptic operator $D$ on $X$ with "non commutative symbol" $[\sigma_D]\in K_0(C^*(\mathcal{T}_{nc}X))$ we have the following cohomological formula for the APS index:
\begin{equation}\label{APScohomology}
Index_{APS}(D)=\int_{\mathscr{N}_{sing}(X)}Ch((\mathscr{CT}([\sigma_{D}])))
\end{equation}
\end{corollary}

Remember that the space $\mathscr{N}_{sing}(X)$ already splits in two, exhibiting in this way the contributions from the interior and from the boundary. The interior contribution looks classic but an explicit comparison between the Thom isomorphism and the Connes-Thom isomorphism is needed. This will be detailed elsewhere. 

In particular, picking up a differential form $\omega_D$ on $\mathscr{N}_{sing}(X)$ representing $Ch(\mathscr{CT}([\sigma_{D}])$, we obtain:
\begin{equation}
Index_{APS}(D)=\int_{\mathscr{N}(X) }\omega_D + \, \int_{D_{\partial}} \omega_D.
\end{equation}
The first integral above involves the restriction of $\omega_D$ to $\mathscr{N}(X)$, which is related to the ordinary principal symbol of $D$. More precisely, the principal symbol $\sigma_{pr}(D)$ of $D$ provides a $K$-theory class of $C^*(A^*(\Gamma(X)))$, that is a compactly supported $K$-theory class of the dual of the Lie algebroid of $\Gamma(X)$ or in other words of the $b$-cotangent bundle ${}^bT^*X$, and by functoriality of both the Chern character and Thom-Connes maps, we have 
 $$ [(\omega_D)|_{\mathscr{N}(X)}]= Ch(\mathscr{CT}([\sigma_{pr}(D)]).$$ 
The second integral can thus be viewed as a correction term, which contains the eta invariant appearing in APS formula and which also depends on the choice of the representative $\omega_D\in Ch(\mathscr{CT}([\sigma_{D}]))$.


 \bigskip We end this Section by showing that the $K$-theory class $\mathscr{CT}([\sigma_{\ubd,D}])\in K^0(\mathscr{N}_{sing}(X))$ associated to any  fully elliptic $b$-operator $D$ of order $1$ can be  explicitly computed. Here $\sigma_{\ubd,D}$ is the unbounded $KK$-element defined in paragraph \ref{fim}.  
Similar computations can be done with $0$-order pseudodifferential operators. Thus, we intend to describe (see Section \ref{CTsection}): 
\begin{equation}\label{to_be_computed}
 \mathscr{CT}([\sigma_{\ubd,D}]) := \mathscr{M}\circ (e_{1,t})_*\circ (e_{0,t})_*^{-1}\circ B (\sigma_{\ubd,D})\in K^0(\mathscr{N}_{sing}(X)).
\end{equation}

We assume that $N=2M$ is even, we identify $\RR^N$ with $\CC^M$, we denote by $\Lambda^*(\CC^M)$ the exterior algebra of $\CC^M$, by $c(v)=v\wedge\cdot -\overline{v}\llcorner\cdot$ the Clifford multiplication by $v$, and we consider the following unbounded representative of the $M^{\text{th}}$ power of Bott class $\beta\in KK(\CC,C_0(\RR^2))$:
\begin{equation}\label{be}
 \beta^M = \left( C_0(\RR^N,\Lambda^*(\CC^M)),1,  c \right)\in \cE_{\ubd}(\CC,C_0(\RR^N))
\end{equation}
with the grading given by even/odd forms.
Then $B(\sigma_{\ubd,D})$ is represented by 
\begin{equation}
 \Sigma_{\ubd,D}:=\left(C^*(\mathcal{T}_{nc}X\times\RR^N,r^*E_{\ad}\otimes\Lambda^*(\CC^M)),1, \Sigma_{\ubd}(D) \right)\in\cE_{\ubd}(\CC,C^*(\mathcal{T}_{nc}X\times\RR^N))
\end{equation}
where we have set 
\begin{equation}\label{Sigma_unbounded}
 \Sigma_{\ubd}(D) := \sigma_{\ubd}(D)\hat{\otimes} 1+1\hat{\otimes}c.
\end{equation}
The latter is a priori a differential  $\mathcal{T}_{nc}X\times\RR^N$-operator, that is an operator on the cartesian product of the groupoid $\mathcal{T}_{nc}X$ with the space $\RR^N$, but a straight computation shows it is also a $\T\ltimes\RR^N$-operator and we prove:
\begin{proposition}
The operator $\Sigma_{\ubd}(D)\in\mathrm{Diff}^1(\T\ltimes\RR^N,r^*E_{\ad}\times \Lambda^*(\CC^M))$ defined above is symmetric, elliptic and gives an unbounded $KK$-element:
\begin{equation}
\Sigma_h:=\left(C^*(\T\ltimes\RR^N,r^*E_{\ad}\otimes\Lambda^*(\CC^M)),1, \Sigma_{\ubd}(D)\right)\in \cE_{\ubd}(\CC,C^*(\T\ltimes\RR^N)),
\end{equation}
which represents the image of $\sigma_{\ubd,D}$ under the Thom-Connes map $\mathscr{CT}_0$:
 $$[\Sigma_h]:= (e_{1,t})_*\circ (e_{0,t})_*^{-1}\circ B (\sigma_{\ubd,D})=\mathscr{CT}_0(\sigma_{\ubd,D})\in KK(\CC,C^*(\T\ltimes\RR^N)).$$ 
 \end{proposition}
\begin{proof}
Let us consider $p:= \Sigma_{\ubd}(D)^2+1= \sigma_{\ubd}(D)^2+ |X|^2+1 \in \mathrm{Diff}^2(\mathcal{T}_{nc}X\times\RR^N,r^*E_{\ad}\times \Lambda^*(\CC^M))$, where we have taken into account the identity $c^2(f)(\gamma,X)=|X|^2.f(\gamma,X)$.  It is self-adjoint and for any $(x,X)$ belonging to the unit space $(\mathcal{T}_{nc}X\times\RR^N)^{(0)}=X_\partial\times\RR^N$, the operator $\Sigma_{\ubd}(D)^2_{(x,X)}+1$ is invertible by classical arguments. Thus, $p$ itself is invertible, which means that $p^{-1}$ exists and belongs to $\Psi^{-2}(\mathcal{T}_{nc}X\times\RR^N,r^*E_{\ad}\times \Lambda^*(\CC^M))$. As a direct consequence, $\Sigma_{\ubd}(D)$ is regular. Moreover, let us consider the groupoid $\overline{\T}\times B^N\rightrightarrows \overline{X_\partial}\times B^N$ where $B^N$ is the compactification of $\RR^N$ by a sphere at infinity and we recall that $\overline{\T}=\T\cup \partial X\times \partial X\times \RR\times\{\varepsilon=1\}\rightrightarrows X\times[0,1]_{\varepsilon}$. The groupoid structure of $\overline{\T}\times B^N$ is the obvious one and it inherits a natural $C^{\infty,0}$-structure in such a way that it contains $\T\times\RR^N$ as an open saturated $C^{\infty,0}$-subgroupoid. We also extend the vector bundles $r^*E_{\ad}$ and $\Lambda^*(\CC^M)$ onto $\overline{X_\partial}$ and $B^N$ without changing the notations. Thanks to the decay of $p^{-1}$ when $\varepsilon\to 1$ or $|X|\to \infty$, we see that the extension of $p^{-1}$ by $0$ gives an element of $\Psi^{-2}(\overline{\T}\times B^N,r^*E_{\ad}\times \Lambda^*(\CC^M))$. Since $\overline{\T}\times B^N$ has a compact unit space, we know by \cite{MP,VassoutJFA} that $p^{-1}$ is a compact morphism of the $C^*(\overline{\T}\times B^N)$-Hilbert module $\overline{\cE}:= C^*(\overline{\T}\times B^N ,r^*E_{\ad}\times \Lambda^*(\CC^M))$. Considering the closed saturated $C^{\infty,0}$-subgroupoid $\partial(\overline{\T}\times B^N)$ given by the equation $z:=(1-\varepsilon)(1/(|X|+1))=0$ and the corresponding Hilbert submodule $\partial\overline{\cE}$, we get an exact sequence of $C^*$-algebras
\begin{equation}
 0\longrightarrow \cK(\cE) \longrightarrow \cK(\overline{\cE}) \overset{\hbox{rest}|_{z=0}}{\longrightarrow} \cK(\partial\overline{\cE})\longrightarrow 0.
\end{equation}
Since $\hbox{rest}|_{z=0}(p^{-1})=0$, we conclude that $p^{-1}$ belongs to $\cK(\cE)$. 

Replacing $h$ by $H=(th)_{t\in[0,1]}$ and extending the previous construction to the groupoid $\T\ltimes\RR^N\times[0,1]\rightrightarrows X_\partial\times\RR^N\times[0,1]$, we get an unbounded Kasparov class  $\Sigma_H:= (C^*(\T\ltimes\RR^N\times[0,1],r^*E_{\ad}\otimes\Lambda^*(\CC^M)),1, \Sigma_{\ubd}(D))\in \Psi^1(\CC,C^*(\T\ltimes\RR^N\times[0,1]))$ such that 
$(e_{0,t})_*([\Sigma_H])=B (\sigma_{\ubd,D})$ and $(e_{1,t})_*([\Sigma_H])=[\Sigma_h]$.
\end{proof}


It remains to apply the Morita equivalence $\mathscr{M}$ between the groupoid $\T\ltimes\RR^N$ and the orbit space $\mathscr{B}_{h_0}=X_{\partial}\times \RR^N/\T\ltimes\RR^N$ to the class $[\Sigma_h]$, since $\mathscr{CT}([\sigma_{\ubd,D}])=\mathscr{M}(\Sigma_h)$.
Following \cite{Macho-Ouchi,Tu04}, the Morita isomorphism $\mathscr{M}:K_0(C^*(\T\ltimes\RR^N))\to K_0(C_0(X_{\partial}\times \RR^N/\T\ltimes\RR^N))$ is given by the Kasparov product with the class $\mathscr{M}$ given by: 
\begin{equation}
 \mathscr{M} = (\cE_\mathscr{M} ,r_*,0) \in KK(C^*(\T\ltimes\RR^N),C_0(X_{\partial}\times \RR^N/\T\ltimes\RR^N))).
\end{equation}
whose ingredients are recalled below.


The space $C_c(X_{\partial}\times \RR^N)$ is in a natural way a right $C_0(X_{\partial}\times \RR^N/\T\ltimes\RR^N)$-module: 
\begin{equation}\label{right-mod}
 \forall \xi\in C_c(X_{\partial}\times \RR^N),\forall a\in C_0(X_{\partial}\times \RR^N/\T\ltimes\RR^N),\quad    \xi.a(z)=  \xi(z) a([z]),
\end{equation}
and using the right Haar system $\lambda$ on $\T\ltimes\RR^N$ previously introduced, the following formula: 
\begin{equation} 
  <\xi,\eta>([z]) = \int_{(\T\ltimes\RR^N)_z}\overline{\xi(r(\gamma))}\eta(r(\gamma))d\lambda_z(\gamma),
\end{equation}
 also defines a $C_0(X_{\partial}\times \RR^N/\T\ltimes\RR^N)$-prehilbertian module structure on $C_c(X_{\partial}\times \RR^N)$. Then $\cE_\mathscr{M}$ is defined to be the Hilbert module completion of $C_c(X_{\partial}\times \RR^N)$: 
\begin{equation}\label{morita_bimodule}
 \cE_\mathscr{M} = \overline{C_c(X_{\partial}\times \RR^N)}^{<,>}.
\end{equation}
 Finally, the representation $r_*$ comes from the target map as follows:  
\begin{equation}\label{left-mod}
 \forall b\in C^{\infty,0}_c(\T\ltimes\RR^N),\forall \xi\in C_c(X_{\partial}\times \RR^N),\quad r_*(b).\xi(z):=(b.r^*(\xi))(z), 
\end{equation}
where the $\cdot$ above just denotes the convolution product in $C^{\infty,0}(\T\ltimes\RR^N)$. The map $r_*$ clearly extends into a $*$-homomorphism 
 $r_* : C^*( \T\ltimes\RR^N)\to \cK(\cE_\mathscr{M})$. 

The Kasparov product $\Sigma_h\underset{C^*(\T\ltimes\RR^N)}{\otimes}\mathscr{M}$ is formally given by:
\begin{equation}\label{answer_1}
\left(C^*(\T\ltimes\RR^N,r^*E)\underset{r_*}{\otimes}\cE_{\mathscr{M}},1,\Sigma_{\ubd}(D)\widehat{\otimes}1 \right)\in KK(\CC,C_0(X_{\partial}\times \RR^N/\T\ltimes\RR^N)),
\end{equation}
where we have noted $E:=E_{\ad}\otimes\Lambda^*(\CC^M)$. It is unitarly equivalent to the following class:
\begin{equation}\label{answer_2}
\left(\cE_E,1,r_*(\Sigma_{\ubd}(D)) \right)\in KK(\CC,C_0(X_{\partial}\times \RR^N/\T\ltimes\RR^N)),
\end{equation}
where $\cE_E$ is the hilbert module completion of $C_c(X_{\partial}\times \RR^N,E)$ with respect to the $C_0(X_{\partial}\times \RR^N/\T\ltimes\RR^N)$-prehilbertian module structure given by 
\begin{equation}
  <\xi,\eta>([z]) = \int_{(\T\ltimes\RR^N)_z} <\xi(r(\gamma)),\eta(r(\gamma))>_{E_{r(\gamma)}}d\lambda_z(\gamma),
\end{equation}
and for any pseudodifferential operator $P$ on a groupoid $G$ acting on the sections of a bundle $F$ over $G^{(0)}$, the operator  $r_*(P)$ acting on $C^\infty_c(G^{(0)},F)$ is defined by, under appropriate properness conditions on the support of $P$ which are fulfilled in our case:  $r_*(P)(f)(z)=P(f\circ r)(z)$.

The unitary equivalence between \eqref{answer_1} and \eqref{answer_2} comes from the map:
\begin{alignat}{2}\label{mod_identification}
 U :& \ & C_c( \T\ltimes\RR^N,r^*E)\otimes C_c(X_{\partial}\times \RR^N) & \longrightarrow  C_c(X_{\partial}\times \RR^N,E) \\
     & \  &  a\otimes f  & \longrightarrow r_*(a)f 
\end{alignat}
which extends into the required unitary between the Hilbert modules $C^*(\T\ltimes\RR^N,E)\underset{r_*}{\otimes}\cE_{\mathscr{M}}$ and $\cE_E$.

The $K$-theory element \eqref{answer_2}, which represents $\mathscr{CT}([\sigma_{\ubd,D}])$ up to the homeomorphism between the orbit space $X_{\partial}\times \RR^N/\T\ltimes\RR^N$ and the singular normal space $\mathscr{N}_{sing}(X)$, is simpler than it looks at first glance: it is given by a family of elliptic differential operators along the orbits of $X_{\partial}\times \RR^N$ parametrized by the orbit space. Moreover, when  applying the homeomorphism  $\widetilde{q_\partial}:X_{\partial}\times \RR^N/\T\ltimes\RR^N\overset{\simeq}{\to} \mathscr{N}_{sing}(X)$ of lemma \ref{lemmaqbord}, we get a family $(\Sigma_o)_{o\in \mathscr{N}_{sing}(X)}$ which has a simple expression in terms of the symbol of $D$ over  $N(X)\subset \mathscr{N}_{sing}(X)$  and in terms of the indicial operator of $D$ over $(0,1)\times\RR^{N-1}\subset \mathscr{N}_{sing}(X)$. One can then add to this family the de Rham differential of the manifold $\mathscr{N}_{sing}(X)$ in order to get a superconnection \cite{Bisinv} and introduce a differential form on $\mathscr{N}_{sing}(X)$ using the trace  $\overline{tr}$ studied in \cite{Meleta} and the heat operator of the square of the superconnection following the method of \cite{Quillen85,Bisinv,BGV}: we expect it to give the Chern character of $\mathscr{CT}([\sigma_{\ubd,D}])$ and then an explicit form for the index formula \eqref{APScohomology}. The details will be worked out in a forthcoming paper and in the more general framework of manifolds with corners.

%
%

\appendix
\section{Deformation to the normal cone functor and tangent groupoids}\label{dncappendix}

The tangent groupoid is a particular case of a geometric construction that we describe here.

Let $M$ be a $C^\infty$ manifold and $X\subset M$ be a $C^\infty$ submanifold. We denote
by $\Nb_{X}^{M}$ the normal bundle to $X$ in $M$.
We define the following set
\begin{align}
\mathscr{D}_{X}^{M}:= \Nb_{X}^{M} \times {0} \bigsqcup M \times \RR^* 
\end{align} 
The purpose of this Section is to recall how to define a $C^\infty$-structure in $\mathscr{D}_{X}^{M}$. This is more or less classical, for example
it was extensively used in \cite{HS}.

Let us first consider the case where $M=\RR^p\times \RR^q$ 
and $X=\RR^p \times \{ 0\}$ ( here we
identify  $X$ canonically with $ \RR^p$). We denote by
$q=n-p$ and by $\mathscr{D}_{p}^{n}$ for $\mathscr{D}_{\RR^p}^{\RR^n}$ as above. In this case
we   have that $\mathscr{D}_{p}^{n}=\RR^p \times \RR^q \times \RR$ (as a
set). Consider the 
bijection  $\psi: \RR^p \times \RR^q \times \RR \rightarrow
\mathscr{D}_{p}^{n}$ given by 
\begin{equation}\label{psi}
\psi(x,\xi ,t) = \left\{ 
\begin{array}{cc}
(x,\xi ,0) &\mbox{ if } t=0 \\
(x,t\xi ,t) &\mbox{ if } t\neq0
\end{array}\right.
\end{equation}
whose  inverse is given explicitly by 
$$
\psi^{-1}(x,\xi ,t) = \left\{ 
\begin{array}{cc}
(x,\xi ,0) &\mbox{ if } t=0 \\
(x,\frac{1}{t}\xi ,t) &\mbox{ if } t\neq0
\end{array}\right.
$$
We can consider the $C^\infty$-structure on $\mathscr{D}_{p}^{n}$
induced by this bijection.

We  now switch to the general case. A local chart 
$(\mathcal{U},\phi)$ in $M$ is said to be a $X$-slice if 
\begin{itemize}
\item[1)]$\phi : \mathcal{U}  \rightarrow U \subset \RR^p\times \RR^q$ is a diffeomorphsim. 
\item[2)]  If  $V =U \cap (\RR^p \times \{ 0\})$, then
$\phi^{-1}(V) =   \mathcal{U} \cap X$ , denoted by $\mathcal{V}$.
\end{itemize}
With this notation, $\mathscr{D}_{V}^{U}\subset \mathscr{D}_{p}^{n}$ as an
open subset. We may define a function 
\begin{equation}\label{phi}
\tilde{\phi}:\mathscr{D}_{\mathcal{V}}^{\mathcal{U}} \rightarrow \mathscr{D}_{V}^{U} 
\end{equation}
in the following way: For $x\in \mathcal{V}$ we have $\phi (x)\in \RR^p
\times \{0\}$. If we write 
$\phi(x)=(\phi_1(x),0)$, then 
$$ \phi_1 :\mathcal{V} \rightarrow V \subset \RR^p$$ 
is a diffeomorphism. We set 
$\tilde{\phi}(v,\xi ,0)= (\phi_1 (v),d_N\phi_v (\xi ),0)$ and 
$\tilde{\phi}(u,t)= (\phi (u),t)$ 
for $t\neq 0$. Here 
$d_N\phi_v: N_v \rightarrow \RR^q$ is the normal component of the
 derivative $d\phi_v$ for $v\in \mathcal{V}$. It is clear that $\tilde{\phi}$ is
 also a  bijection (in particular it induces a $C^{\infty}$ structure on $\mathscr{D}_{\mathcal{V}}^{\mathcal{U}}$). 
 
Let us define, with the same notations as above, the following set 
\begin{equation}\label{openOmega}
\Omega_{V}^{U}=\{(x,\xi,t)\in 
\RR^p \times \RR^q \times \RR: (x,t\cdot \xi)\in U \}.
\end{equation}
which is an open subset of $\RR^p \times \RR^q \times [0,1]$ and thus
a $\ci$ manifold (with border). 
It is immediate that $\mathscr{D}_{V}^{U}$ is diffeomorphic to $\Omega_{V}^{U}$
through the restriction of $ \Psi$, used in (\ref{psi}). 
Now we consider an atlas 
$ \{ (\mathcal{U}_{\alpha},\phi_{\alpha}) \}_{\alpha \in \Delta}$ of $M$
 consisting of $X-$slices. It is clear that 
\begin{equation}\label{atlasdcn} 
\mathscr{D}_{X}^{M}= \cup_{\alpha \in
 \Delta}\mathscr{D}_{\mathcal{V}_{\alpha}}^{\mathcal{U}_{\alpha}}
\end{equation}
and if we take $\mathscr{D}_{\mathcal{V}_{\alpha}}^{\mathcal{U}_{\alpha}} 
\stackrel{\varphi_{\alpha}}{\rightarrow } 
\Omega_{V_{\alpha}}^{U_{\alpha}}$ defined as the composition
$$\mathscr{D}_{\mathcal{V}_{\alpha}}^{\mathcal{U}_{\alpha}} \stackrel{\tilde{\phi_{\alpha}}}{\rightarrow}
\mathscr{D}_{V_{\alpha}}^{U_{\alpha}}
\stackrel{\Psi_{\alpha}^{-1}}{\rightarrow} 
\Omega_{V_{\alpha}}^{U_{\alpha}} $$
then we have (proposition 3.1 in \cite{Ca2}).

\begin{proposition}\label{dncatlas}
$ \{ (\mathscr{D}_{\mathcal{V}_{\alpha}}^{\mathcal{U}_{\alpha}},\varphi_{\alpha})
  \} _{\alpha \in \Delta }$ is a $C^\infty$ atlas over
  $\mathscr{D}_{X}^{M}$.
\end{proposition}

\begin{definition}[Deformation to the normal cone]
Let $X\subset M$ be as above. The set
$\mathscr{D}_{X}^{M}$ equipped with the  $C^{\infty}$ structure
induced by the atlas of  $X$-slices is called
 the deformation to the  normal cone associated  to   the embedding
$X\subset M$. 
\end{definition}


One important feature about the deformation to the normal cone is the functoriality. More explicitly,  let
 $f:(M,X)\rightarrow (M',X')$
be a   $C^\infty$ map   
$f:M\rightarrow M'$  with $f(X)\subset X'$. Define 
$ \mathscr{D}(f): \mathscr{D}_{X}^{M} \rightarrow \mathscr{D}_{X'}^{M'} $ by the following formulas: \begin{enumerate}
\item[1)] $\mathscr{D}(f) (m ,t)= (f(m),t)$ for $t\neq 0$, 

\item[2)]  $\mathscr{D}(f) (x,\xi ,0)= (f(x),d_Nf_x (\xi),0)$,
where $d_Nf_x$ is by definition the map
\[  (\Nb_{X}^{M})_x 
\stackrel{d_Nf_x}{\longrightarrow}  (\Nb_{X'}^{M'})_{f(x)} \]
induced by $ T_xM 
\stackrel{df_x}{\longrightarrow}  T_{f(x)}M'$.
\end{enumerate}
 Then we have, (proposition 3.4 in \cite{Ca2}), 
 \begin{proposition}\label{dncfunctoriality}
 $\mathscr{D}(f):\mathscr{D}_{X}^{M} \rightarrow \mathscr{D}_{X'}^{M'}$ is a $C^\infty$-map. In the language of categories, the deformation to the normal cone  construction defines a functor
\begin{equation}\label{fundnc}
\mathscr{D}: \mathscr{C}_2^{\infty}\longrightarrow \mathscr{C}^{\infty} ,
\end{equation}
where $\mathscr{C}^{\infty}$ is the category of $C^\infty$-manifolds and $\mathscr{C}^{\infty}_2$ is the category of pairs of $C^\infty$-manifolds.
\end{proposition}

In \cite{Pat2000}, Paterson properly defined the notion of continuous family groupoids, for which all above considerations and concepts apply immediately. For more details on this the reader might consult Paterson original paper or \cite{LMNpdo} where the authors also developed the appropriate pseudodifferential calculus. In particular, we can define the Connes tangent groupoid in this context:

\begin{definition}\label{defgtan}[Tangent groupoid of a continuous family groupoid]
Let $\gr \rightrightarrows \go $ be a continuous family groupoid. $\it{The\, tangent\,
groupoid}$ associated to $\gr$ is the groupoid that has 
\[
\mathscr{D}_{\go}^{\gr} = \Nb_{\go}^{\gr} \times \{0\}\bigsqcup \gr\times \RR^*
\]
 as the set of arrows and  $\go \times \RR$ as the units, with:
 \begin{enumerate}
\item  $s^T(x,\eta ,0) =(x,0)$ and $r^T(x,\eta ,0) =(x,0)$ at $t=0$.
\item   $s^T(\gamma,t) =(s(\gamma),t)$ and $r^T(\gamma,t)
  =(r(\gamma),t)$ at $t\neq0$.
\item  The product is given by
  $m^T((x,\eta,0),(x,\xi,0))=(x,\eta +\xi ,0)$ and  $m^T((\gamma,t), 
  (\beta ,t))= (m(\gamma,\beta) , t)$ if $t\neq 0 $ and 
if $r(\beta)=s(\gamma)$.
\item The unit map $u^T:\go \rightarrow \gr^T$ is given by
 $u^T(x,0)=(x,0)$ and $u^T(x,t)=(u(x),t)$ for $t\neq 0$.
\end{enumerate}
We denote $\gr^{T}= \mathscr{D}_{\go}^{\gr}$ and $A\gr =\Nb_{\go}^{\gr}  $ the normal vector bundle over $\gr^{(0)}$ associated to the embedding of $\go$ into $\gr$ as units. Then we have a family of continuous family groupoids parametrized by $\RR$, which itself is a continuous family groupoid
\[
\gr^T = A\gr \times \{0\}  \bigsqcup \gr\times \RR^* \rightrightarrows \go\times \RR.
\]
\end{definition} 

The vector bundle $A\gr =\Nb_{\go}^{\gr} \to \go $ is called the Lie algebroid of $\gr$ (In this paper we will not use its Lie algebroid structure so we will not enter into these details). It can be identified with the vector bundle 
\begin{equation}\label{AGds}
A\gr=Ker\, ds|_{\go},
\end{equation}
via the short splitting exact sequence of vector bundles over $\go$:
$$0\to T\go\to T_{\go}\gr\to A\gr \to 0.$$
In practice it is useful to use such identification to use the explicit fiber decomposition
$$A\gr=\bigsqcup_{x\in\go}T_x\gr_x.$$

As a consequence of the functoriality of the deformation to the normal cone,
one can show that the tangent groupoid is in fact a continuous family groupoid compatible with the continuous family groupoid structures of $\gr$ and $A\gr$ (considered as a continuous family groupoid with its vector bundle structure). 

Before enouncing the following proposition, let us recall the following fact: Given a groupoid (strict) morphism
$$\gr\stackrel{f}{\longrightarrow} \hr$$
there is an induced continuous vector bundle morphism
\begin{equation}\label{AGmorphism}
A(f):A\gr\to A\hr
\end{equation}
given by derivation in the normal direction. Indeed, using the intrinsic definitions of $A\gr=\Nb_{\go}^{\gr}$ and $A\hr=\Nb_{\go}^{\gr}$ we have that $df$ passes to the quotient since $f(\go)\subset \ho$. In terms of the identification (\ref{AGds}), $A(f)$ can be defined by
$$A(f)(x,\xi)=(f(x),d_xf(\xi)).$$
for $(x,\xi)\in T_x\gr_x$
\begin{proposition}\label{tanproperaction}
Let $\gr$ be a continuous family groupoid together with an injective continuous family groupoid morphism $\gr \stackrel{h}{\longrightarrow} \RR^N$. Consider also the induced infinitesimal continuous family groupoid morphism $A(\gr) \stackrel{A(h)}{\longrightarrow} \RR^N$. Assume that $A(h)$ is also injective and that both semi-direct groupoids, $\gr_h$ and $A(\gr)_{A(h)}$ are free and proper. Then the induced morphism 
$\gr^{tan} \stackrel{h^T}{\longrightarrow} \RR^N$ gives a free proper semi-direct groupoid as well.
\end{proposition}

\begin{proof}
We will use again properness caracterization $(ii)$ of proposition 2.14 in \cite{Tu04}. In particular we have to verify only property (A)  (of Section \ref{sectionembedding} above), that is, in our case we have to check that the map
\begin{equation}
\xymatrix{
\gr^{tan}\times \RR^N\ar[rr]^-{(r,s)}&&(\go\times [0,1]\times \RR^N)\times (\go\times [0,1]\times \RR^N)
}
\end{equation}
given by 
\begin{equation}\label{(r,s)tan}
\begin{cases}
((\gamma,\epsilon),X)\mapsto ((t(\gamma),\epsilon, X),(s(\gamma),\epsilon,X+\frac{h(\gamma)}{\epsilon}))
\\
((x,\xi),X)\mapsto ((x,0, X),(x,0,X+A(h)(x,\xi)))
\end{cases}
\end{equation}
is closed.

Let $(A_n)_n:=(\tilde{\gamma_n},X_n)_n$ a sequence in $\gr^{tan}\times \RR^N$ such that 
\begin{equation}\label{convergencegammatan}
lim_{n\to \infty}(r,s)(\tilde{\gamma_n},X_n)=P
\end{equation}
with $P$ a point in $(\go\times [0,1]\times \RR^N)\times (\go\times [0,1]\times \RR^N)$. It is enough to justify that there is a subsequence of $(A_n)_n$ converging to an antecedent of $P$: 
The point $P$ is of the form $((x,\epsilon_1,X),(y,\epsilon_2,Y))$. The first consequence of (\ref{convergencegammatan}) is that $\epsilon_1=\epsilon_2$, hence $P=((x,\epsilon,X),(y,\epsilon,Y))$

We will separate the analysis in two cases:
\begin{itemize}
\item[$(a)$] The case $\epsilon \neq 0$: By the explicit form of (\ref{(r,s)tan}), we can assume (or there is a subsequence) that $(A_n)_n\subset \gr\times (0,1]\times \RR^N$, {i.e.}, that the elements os the sequence are of the form $A_n=(\gamma_n,\epsilon_n,X_n)$ with $\epsilon\neq 0$. But then we have the following convergences: $t(\gamma_n)\to x$, $\epsilon_n\to \epsilon$, $X_n\to X$, $s(\gamma_n)\to y$ and  $X_n+\frac{h(\gamma_n)}{\epsilon_n}$.

In particular we obtain that $h(\gamma_n)\to \epsilon\cdot (Y-X)$, and since $\gr_h$ is proper we have that there is a subsequence of $(\gamma_{n_k})_k$ of $(\gamma_n)_n$ and a $\gamma\in \gr$ such that $\gamma_{n_k}\to \gamma$ with $t(\gamma)=x$, $s(\gamma)=y$ and $h(\gamma)=\epsilon\cdot (Y-X)$. In particular, letting $A=(\gamma,\epsilon,X)$ we have that 
$A_{n_k}\to A$  and $(r,s)(A)=P$.

\item[$(b)$] The case $\epsilon=0$: In this case we have two subcases: 
\begin{itemize}
\item[$(b_1)$] There is a subsequence of $(A_n)_n$ entirely contained in 
$A(\gr)\times \RR^N$. In this case we might assume that $A_n=((x_n,\xi_n),X_n)\in A_{x_n}(\gr)\times \RR^N$. Then, (\ref{convergencegammatan}) implies that $x_n\to x=y$, $X_n\to X$ and $X_n+A(h)(x_n,\xi_n)\to Y$. In particular, $A(h)(x_n,\xi_n)\to Y-X$ and since $A(\gr)_{A(h)}$ is proper we have that there is a subsequence $(x_{n_k},\xi_{n_k})$ converging in $A(\gr)$ to an element $(x,\xi)\in A_x(\gr)$. Then letting $A=((x,\xi),X)$ we have that 
$A_{n_k}\to A$ and $(r,s)(A)=P$.
\item[$(b_2)$] There is a subsequence of $(A_n)_n$ entirely contained in 
$\gr \times (0,1] \times \RR^N$. In this case we might assume that $A_n=(\gamma_n,\epsilon,X_n)\in \gr\times (0,1]\times \RR^N$. Then, (\ref{convergencegammatan}) implies that $t(\gamma_n)\to x$, $\epsilon_n\to 0$, $X_n\to X$, $s(\gamma_n)\to y$ and $X_n+\frac{h(\gamma_n)}{\epsilon_n}\to Y$. This implies $\frac{h(\gamma_n)}{\epsilon_n}\to Y-X$ in $\RR^N$ or in other words $(h(\gamma_n),\epsilon_n)\to (Y-X,0)$ in $(\RR^N)^{tan}$. In particular we have also that $h(\gamma_n)\to 0$ in $\RR^N$ and since $\gr_h$ is proper and $h$ is injective, we deduce that $x=y$ and that there is a subsequence $(\gamma_{n_k})_k$ of $(\gamma_n)_n$ such that $\gamma_{n_k}\to x$. Now, from the injectivity of $A(h)$ and the fact that $\frac{h(\gamma_n)}{\epsilon_n}\to Y-X$ we deduce that there is an unique $\xi \in A_x(\gr)$ such that $A(h)(x,\xi)=Y-X$. Finally, letting $A=((x,\xi), X)$ we have that $A_{n_k}\to A$ and $(r,s)(A)=P$.
\end{itemize}
\end{itemize}

\end{proof}

\bibliographystyle{plain} 
\bibliography{clmETA6} 

\begin{thebibliography}{10}

\bibitem{APS1}
M.~F. Atiyah, V.~K. Patodi, and I.~M. Singer.
\newblock Spectral asymmetry and {R}iemannian geometry. {I}.
\newblock {\em Math. Proc. Cambridge Philos. Soc.}, 77:43--69, 1975.

\bibitem{APS2}
M.~F. Atiyah, V.~K. Patodi, and I.~M. Singer.
\newblock Spectral asymmetry and {R}iemannian geometry. {II}.
\newblock {\em Math. Proc. Cambridge Philos. Soc.}, 78(3):405--432, 1975.

\bibitem{APS3}
M.~F. Atiyah, V.~K. Patodi, and I.~M. Singer.
\newblock Spectral asymmetry and {R}iemannian geometry. {III}.
\newblock {\em Math. Proc. Cambridge Philos. Soc.}, 79(1):71--99, 1976.

\bibitem{AS}
M.~F. Atiyah and I.~M. Singer.
\newblock The index of elliptic operators. {I}.
\newblock {\em Ann. of Math. (2)}, 87:484--530, 1968.

\bibitem{AS3}
M.~F. Atiyah and I.~M. Singer.
\newblock The index of elliptic operators. {III}.
\newblock {\em Ann. of Math. (2)}, 87:546--604, 1968.

\bibitem{bj}
S.~Baaj and P.~Julg.
\newblock Th\'eorie bivariante de {K}asparov et op\'erateurs non born\'es dans
  les {$C^{\ast} $}-modules hilbertiens.
\newblock {\em C. R. Acad. Sci. Paris S\'er. I Math.}, 296:875--878, 1983.

\bibitem{BGV}
N.~Berline, E.~Getzler, and M.~Vergne.
\newblock {\em Heat kernels and {D}irac operators}, volume 298 of {\em
  comprehensive studies in mathematics}.
\newblock Springer-Verlag, 1991.

\bibitem{Bisinv}
J.-M. Bismut.
\newblock The {A}tiyah-{S}inger index theorem for families of {D}irac
  operators: two heat equation proofs.
\newblock {\em Invent. Math.}, 83(1):91--151, 1985.

\bibitem{Ca2}
P.~Carrillo~Rouse.
\newblock A {S}chwartz type algebra for the tangent groupoid.
\newblock In {\em {$K$}-theory and noncommutative geometry}, EMS Ser. Congr.
  Rep., pages 181--199. Eur. Math. Soc., Z\"urich, 2008.

\bibitem{CMont}
P.~Carrillo~Rouse and B.~Monthubert.
\newblock An index theorem for manifolds with boundary.
\newblock {\em C. R. Acad. Sci. Paris, Ser. I (2009),
  doi:10.1016/j.crma.2009.10.021}, 2009.

\bibitem{Concg}
A.~Connes.
\newblock {\em Noncommutative geometry}.
\newblock Academic Press Inc., San Diego, CA, 1994.

\bibitem{DL05}
C.~Debord and J.-M. Lescure.
\newblock {$K$}-duality for pseudomanifolds with isolated singularities.
\newblock {\em J. Funct. Anal.}, 219(1):109--133, 2005.

\bibitem{DL09}
C.~Debord and J.-M. Lescure.
\newblock {$K$}-duality for stratified pseudomanifolds.
\newblock {\em Geom. Topol.}, 13(1):49--86, 2009.

\bibitem{DLN}
C.~Debord, J.-M. Lescure, and V.~Nistor.
\newblock Groupoids and an index theorem for conical pseudo-manifolds.
\newblock {\em J. Reine Angew. Math.}, 628:1--35, 2009.

\bibitem{DLR}
C.~Debord, J.-M. Lescure, and F.~Rochon.
\newblock Pseudodifferential operators on manifolds with fibred corners.
\newblock arXiv:1112.4575.

\bibitem{ENN93}
G.~A. Elliott, T.~Natsume, and R.~Nest.
\newblock The {H}eisenberg group and {$K$}-theory.
\newblock {\em $K$-Theory}, 7(5):409--428, 1993.

\bibitem{getzler}
E.~Getzler.
\newblock Pseudodifferential operators on supermanifolds and the
  {A}tiyah-{S}inger index theorem.
\newblock {\em Comm. Math. Phys.}, 92:163--178, 1983.

\bibitem{HS}
M.~Hilsum and G.~Skandalis.
\newblock Morphismes {$K$}-orient\'es d'espaces de feuilles et fonctorialit\'e
  en th\'eorie de {K}asparov (d'apr\`es une conjecture d'{A}. {C}onnes).
\newblock {\em Ann. Sci. \'Ecole Norm. Sup. (4)}, 20(3):325--390, 1987.

\bibitem{LR}
N.~P. Landsman and B.~Ramazan.
\newblock Quantization of {P}oisson algebras associated to {L}ie algebroids.
\newblock In {\em Groupoids in analysis, geometry, and physics ({B}oulder,
  {CO}, 1999)}, volume 282 of {\em Contemp. Math.}, pages 159--192. Amer. Math.
  Soc., 2001.

\bibitem{LMNpdo}
R.~Lauter, B.~Monthubert, and V.~Nistor.
\newblock Pseudodifferential analysis on continuous family groupoids.
\newblock {\em Doc. Math.}, 5:625--655 (electronic), 2000.

\bibitem{lmn}
R.~Lauter, B.~Monthubert, and V.~Nistor.
\newblock Spectral invariance for certain algebras of pseudodifferential
  operators.
\newblock {\em J. Inst. Math. Jussieu}, 4:405--442, 2005.

\bibitem{Les09}
J.-M. Lescure.
\newblock Elliptic symbols, elliptic operators and {P}oincar\'e duality on
  conical pseudomanifolds.
\newblock {\em J. K-Theory}, 4(2):263--297, 2009.

\bibitem{Macho-Ouchi}
M.~Macho-Stadler and M.~O'uchi.
\newblock Correspondences and groupoids.
\newblock In {\em Proceedings of the {IX} {F}all {W}orkshop on {G}eometry and
  {P}hysics ({V}ilanova i la {G}eltr\'u, 2000)}, volume~3 of {\em Publ. R. Soc.
  Mat. Esp.}, pages 233--238. R. Soc. Mat. Esp., Madrid, 2001.

\bibitem{MelNis}
R.~Melrose and V.~Nistor.
\newblock {$K$}-theory of {$C^*$}-algebras of {$b$}-pseudodifferential
  operators.
\newblock {\em Geom. Funct. Anal.}, 8(1):88--122, 1998.

\bibitem{Mel}
R.B. Melrose.
\newblock {\em The {A}tiyah-{P}atodi-{S}inger index theorem}, volume~4 of {\em
  Research Notes in Mathematics}.
\newblock A K Peters Ltd., Wellesley, MA, 1993.

\bibitem{Meleta}
R.B. Melrose.
\newblock The eta invariant and families of pseudodifferential operators.
\newblock {\em Math. Res. Lett.}, 2(5):541--561, 1995.

\bibitem{MM}
I.~Moerdijk and J.~Mr{\v{c}}un.
\newblock {\em Introduction to foliations and {L}ie groupoids}, volume~91 of
  {\em Cambridge Studies in Advanced Mathematics}.
\newblock Cambridge University Press, Cambridge, 2003.

\bibitem{Mont}
B.~Monthubert.
\newblock Groupoids and pseudodifferential calculus on manifolds with corners.
\newblock {\em J. Funct. Anal.}, 199(1):243--286, 2003.

\bibitem{Mont-Banach}
B.~Monthubert.
\newblock Contribution of noncommutative geometry to index theory on singular
  manifolds.
\newblock In {\em Geometry and topology of manifolds}, volume~76 of {\em Banach
  Center Publ.}, pages 221--237. Polish Acad. Sci., Warsaw, 2007.

\bibitem{MP}
B.~Monthubert and F.~Pierrot.
\newblock Indice analytique et groupo\"\i des de {L}ie.
\newblock {\em C. R. Acad. Sci. Paris S\'er. I Math.}, 325(2):193--198, 1997.

\bibitem{Mr}
J.~Mr{\v{c}}un.
\newblock Functoriality of the bimodule associated to a {H}ilsum-{S}kandalis
  map.
\newblock {\em $K$-Theory}, 18(3):235--253, 1999.

\bibitem{NWX}
V.~Nistor, A.~Weinstein, and P.~Xu.
\newblock Pseudodifferential operators on differential groupoids.
\newblock {\em Pacific J. Math.}, 189(1):117--152, 1999.

\bibitem{Pat}
A.L.T. Paterson.
\newblock {\em Groupoids, inverse semigroups, and their operator algebras},
  volume 170 of {\em Progress in Mathematics}.
\newblock Birkh\"auser Boston Inc., Boston, MA, 1999.

\bibitem{Pat2000}
A.L.T. Paterson.
\newblock Continuous family groupoids.
\newblock {\em Homology, Homotopy and Applications}, 2:89--104, 2000.

\bibitem{Quillen85}
D.~Quillen.
\newblock Superconnections and the {C}hern character.
\newblock {\em Topology}, 24:89--95, 1985.

\bibitem{Ren}
J.~Renault.
\newblock {\em A groupoid approach to {$C\sp{\ast} $}-algebras}, volume 793 of
  {\em Lecture Notes in Mathematics}.
\newblock Springer, Berlin, 1980.

\bibitem{Savin2005}
A.~Savin.
\newblock Elliptic operators on manifolds with singularities and {K}-homology.
\newblock {\em K-theory}, 34:71--98, 2005.

\bibitem{Tu04}
J.-L. Tu.
\newblock Non-{H}ausdorff groupoids, proper actions and {$K$}-theory.
\newblock {\em Doc. Math.}, 9:565--597 (electronic), 2004.

\bibitem{TXL}
J.-L. Tu, P.~Xu, and C.~Laurent-Gengoux.
\newblock Twisted {$K$}-theory of differentiable stacks.
\newblock {\em Ann. Sci. \'Ecole Norm. Sup. (4)}, 37(6):841--910, 2004.

\bibitem{VassoutJFA}
S.~Vassout.
\newblock Unbounded pseudodifferential calculus on {L}ie groupoids.
\newblock {\em J. Funct. Anal.}, 236:161--200, 2006.

\end{thebibliography}

\end{document}